\documentclass[a4paper, 11pt]{article}

\usepackage[latin1]{inputenc}
\usepackage{cite}
\usepackage{setspace}
\usepackage{amssymb}
\usepackage{amsthm,mathtools}
\mathtoolsset{showonlyrefs}
\usepackage{empheq}
\usepackage{graphicx}
\usepackage[margin=3cm]{geometry}
\usepackage{algorithmic, algorithm}
\usepackage{array}
\usepackage{multirow}
\usepackage{color}
\usepackage{comment}
\usepackage{booktabs,caption,subcaption}
\usepackage[font=footnotesize,labelformat=simple]{subcaption}
\usepackage{tikz}
\usetikzlibrary{spy,calc}
\usepackage{hyperref}

\DeclareCaptionLabelSeparator{periodspace}{.\ }
\newtheorem{theorem}{Theorem}[section]

\newtheorem{proposition}[theorem]{Proposition}
\newtheorem{lemma}[theorem]{Lemma}
\newtheorem{remark}[theorem]{Remark}

\newcommand{\Ss}{\mathbb{S}}
\newcommand{\NN}{\mathbb{N}}

\newcommand{\tT}{\mathrm{T}}
\DeclareMathOperator*{\argmin}{arg\,min}

\DeclareMathOperator{\grad}{grad}

\DeclareMathOperator{\sgn}{sgn}

\DeclareMathOperator{\tr}{tr}

\DeclareMathOperator{\Log}{Log}
\DeclareMathOperator{\Exp}{Exp}

\title{Restoration of Manifold-Valued Images by Half-Quadratic Minimization}
\hypersetup{pdfauthor={Ronny Bergmann, Raymond H. Chan, Ralf Hielscher, 
Johannes Persch, Gabriele Steidl},
	pdftitle={Restoration of Manifold-Valued Images by Half-Quadratic Minimization},
	pdfsubject={Preprint},%
	pdfcreator = {pdflatex and TextMate},%
	pdfkeywords={},
	plainpages=false, pdfstartview=FitH, pdfview=FitH, pdfpagemode=UseOutlines,%
	bookmarksnumbered=true,bookmarksopen=false,bookmarksopenlevel=0,%
	colorlinks=true,linkcolor=black,citecolor=black,urlcolor=black%
}
\author{
Ronny Bergmann\thanks{University of Kaiserslautern, Dept. of Mathematics, Paul-Ehrlich-Str.~31, 67663 Kaiserslautern, Germany, $\{$bergmann, persch, steidl$\}$@mathematik.uni-kl.de},%
\and
Raymond H. Chan\thanks{Chinese University of Hong Kong, Dept. of Mathematics, Hong Kong, China},%
\and
Ralf Hielscher\footnote{University of  Chemnitz, Faculty of  Mathematics, Reichenhainer Str. 39, 09107 Chemnitz, Germany, ralf.hielscher@mathematik.tu-chemnitz.de
},%
\and
Johannes Persch\footnotemark[1],
\and
Gabriele Steidl\footnotemark[1]
}

\date{\today}

\begin{document}
\maketitle

\begin{abstract}
\noindent
The paper addresses the generalization of the half-quadratic minimization method for the restoration of
images having values in a complete, connected Riemannian manifold. 
We recall the half-quadratic minimization method using the notation of the $c$-transform
and adapt the algorithm to our special variational setting. We prove the convergence of the method for
Hadamard spaces. 
Extensive numerical examples for images with values on spheres, in the rotation group $\operatorname{SO}(3)$, and
in the manifold of positive definite matrices demonstrate the excellent performance of the algorithm.
In particular, the method with $\operatorname{SO}(3)$-valued data
shows promising results for the restoration of images obtained
from Electron Backscattered Diffraction which are of interest in material science.
\end{abstract}

\section{Introduction} \label{sec:intro}
Many edge-preserving variational methods for the denoising or inpainting of real-valued images utilize the following model:
let ${\mathcal G} \coloneqq \{1,\ldots,n\} \times \{1,\ldots,m\}$ be the image grid,
$\emptyset \not = {\mathcal V} \subseteq {\mathcal G}$ and 
$
{\mathcal N}(i)^+ \coloneqq\bigl\{ (i_1+1,i_2),(i_1,i_2+1) \bigr\}
$ the set of right and upper neighbors of pixel $i \in {\mathcal G}$,
where we suppose mirror boundary conditions.
From corrupted image values $f\colon {\mathcal V} \rightarrow {\mathbb R}$
we want to restore the original image 
$u_0\colon {\mathcal G} \rightarrow {\mathbb R}$
as a minimizer of one of the following energy functionals
\begin{align}
  \frac12  \sum_{i \in {\mathcal V}} (f_i - u_i)^2 &+ \lambda  \sum_{i \in {\mathcal G}} \sum_{j \in {\mathcal N}(i)^+ } \varphi(\lvert u_i - u_j\rvert), \label{TV_aniso}\\
  \frac12 \sum_{i \in {\mathcal V}} (f_i - u_i)^2  &+ \lambda   \sum_{i \in {\mathcal G}} \varphi \biggl( \Bigl( \sum_{j \in {\mathcal N}(i)^+ } (u_i-u_j)^2\Bigr)^{\frac12} \biggr)\label{TV_iso},
\end{align}
where $\lambda >0$ is a regularization parameter and $\varphi\colon \mathbb R_{\ge 0} \rightarrow \mathbb R_{\ge 0}$.
For ${\mathcal V} = {\mathcal G}$ this is a typical denoising model in the presence of additive Gaussian noise.
Otherwise, the model can be used for inpainting the missing image values in ${\mathcal G} \backslash {\mathcal V}$.
Throughout this paper,  
we consider even functions 
$\varphi\colon\mathbb R \rightarrow \mathbb R_{\ge 0}$.
For $\varphi(t) \coloneqq \lvert t\rvert$, the models~\eqref{TV_aniso}
and~\eqref{TV_iso} are discrete variants 
of the anisotropic and isotropic Rudin-Osher-Fatemi model~\cite{ROF92}, respectively.
Then the regularization term is often referred as anisotropic/isotropic discrete total variation (TV) regularization
in resemblance to its functional analytic counterpart.
\begin{remark}
More generally one may consider ${\mathcal G}$  as vertices of a graph 
with edge set
$
{\mathcal E} \coloneqq \bigl\{ (i,j): i \in {\mathcal G}, j \in {\mathcal N}(i) \bigr\}
$ 
for some appropriate neighborhoods ${\mathcal N}(i)$. 
This is for example useful
in nonlocal means approaches.
Then the regularizing term sums over the edge set ${\mathcal E}$.
For simplicity we restrict our attention to the special neighborhoods ${\mathcal N}(i)^+$ here.
\end{remark}
Starting with the inaugural work~\cite{GR92,GY95}, a huge number of papers
have examined the so-called half-quadratic minimization methods for solving the above restoration problems
with various functions $\varphi$ 
as well as other optimization problems. We only mention the ARTUR algorithm in~\cite{CBAB97}.
Basically, the original problem is reformulated into an augmented one which 
is quadratic with respect to the image and separable with respect to additional auxiliary variables.
Then an alternating minimization process is applied whose steps 
allow an efficient computation.
Half-quadratic minimization is connected with other well-known minimization approaches.
We only mention the relation to EM algorithms~\cite{CI04}, quasi-Newton minimization~\cite{AIG06,NC07}
and gradient linearization algorithms~\cite{NC07}.
A gradient linearization method was used in particular in~\cite{VO96,VO98} 
to minimize an approximate total variation regularization~\eqref{TV_iso} with
$\varphi(t) \coloneqq \sqrt{t^2 + \varepsilon^2}$ for $\varepsilon \ll 1$. 
It was called the ``lagged diffusivity fixed point iteration''
and the authors mention that it amounts 
to apply the multiplicative form of half-quadratic minimization to this $\varphi$.
Finally, there is a relation to iteratively reweighted least squares methods~\cite{Lawson61,DDG10}.
For a convergence analysis of half-quadratic minimization methods for convex functions $\varphi$ 
we refer to~\cite{CBAB97,NN05} and for weaker convergence results for nonconvex $\varphi$ to~\cite{DB98}.

In many applications signals or images having values in a manifold are of interest.
Circle-valued images 
appear in interferometric synthetic aperture radar~\cite{BRF00,DDT11}
and various applications involving the phase of Fourier transformed data.
Images with values in $\mathbb S^2$ play a role when dealing 
with 3D directional information~\cite{KS02,LO14,VO02}
or in the processing of
color images in the chromaticity-brightness (CB) setting~\cite{CKS01}. 
The motion group and the rotation group $\operatorname{SO}(3)$ were considered in
tracking, (scene) motion analysis~\cite{RS10,rosman2012group,TPM08} and in the analysis
of back scatter diffraction data~\cite{BaHiSc11}.
Finally, images with values in the manifold of positive definite matrices appear in 
DT-MRI~\cite{pennec2006riemannian,SSPB07,WFWBB06,WFBW03} 
and whenever covariance matrices are adjusted to image pixels, see, e.g.,~\cite{TPM08}.

Recently a TV-like model for circle-valued images was introduced in~\cite{SC11,SC13}.
For manifold-valued image restoration such an approach was proposed  in~\cite{LSKC13},
where the problem was reformulated as a multilabel optimization problem
which was handled using convex relaxation techniques. 
Another method suggested in~\cite{WDS14} 
employs cyclic and parallel proximal point algorithms 
and does not require labeling and relaxation techniques.
This approach was generalized by including second order differences  for circle-valued images in~\cite{BLSW14,BW15a},
for coupled circle and real-valued images in~\cite{BW15b}
and for Riemannian manifolds in~\cite{BBSW15}.
A restoration method which circumvents the direct work with manifold-valued data by
embedding the matrix manifold in the appropriate Euclidean space and applying
a back projection to the manifold was suggested in~\cite{RTKB14}.
Recently, an iteratively reweighted least squares method for the restoration of manifold-valued images was suggested in~\cite{GS14}.
This method can be seen as multiplicative half-quadratic minimization method for the special function
$\varphi(t) \coloneqq \sqrt{t^2 + \varepsilon^2}$.

In this paper we adopt the idea of half-quadratic minimization for general functions $\varphi$
for the restoration of manifold-valued images.
We prefer the notation of the $c$-transform known from optimal transport to recall
the basic half-quadratic minimization approach.
Then we describe the algorithm for our problems of denoising or inpainting
of manifold-valued images both in the anisotropic and isotropic case.
Here we focus on the multiplicative half-quadratic minimization method.
Convergence of the algorithm can be shown for images having entries in an Hadamard space.
The manifold of positive definite matrices is such an Hadamard manifold.
We provide several applications of the algorithm as the denoising of phase-valued images,
the restoration of color images with disturbed  chromaticity
or of 3D directions, and the improvement of images obtained from electron backscatter diffraction of
a Magnesium sample. 

The outline of the paper is as follows:
In Section~\ref{sec:model} we  propose our variational model and show how to handle it by
the half-quadratic minimization approach.
A convergence proof for the algorithm applied on Hadamard manifolds is given in Section~\ref{sec:hada}. 
Section~\ref{sec:num} shows various numerical examples.
Finally, Appendix~\ref{app:proofs} contains the proofs and Appendix~\ref{app:B} lists 
special quantities necessary for the numerical computations.
\section{Half-Quadratic Minimization} \label{sec:model}
Let $\mathcal M$ be a complete, connected $n$-dimensional Riemannian manifold 
with geodesic distance
$d\colon \mathcal M \times \mathcal M \rightarrow \mathbb R_{\ge 0}$.
Now we consider manifold-valued images. More precisely,
from corrupted image values $f\colon {\mathcal V} \rightarrow {\mathcal M}$
we want to restore the original manifold-valued image
$u_0: {\mathcal G} \rightarrow {\mathcal M}$
as a minimizer of one of the following energy functionals
\begin{align}
 J_1(u) &\coloneqq \frac12 \sum_{i \in {\mathcal V} } d^2(u_i,f_i) + 
 \lambda  \sum_{i \in {\mathcal G}} \sum_{j \in {\mathcal N}(i)^+ } \varphi\bigl(d(u_i,u_j)\bigr), \label{functional_1} \\
 J_2(u) &\coloneqq \frac12 \sum_{i \in {\mathcal V} } d^2(u_i,f_i) + 
 \lambda  \sum_{i \in {\mathcal G}} \varphi \biggl( \Bigl( \sum_{j \in {\mathcal N}(i)^+ } d^2(u_i,u_j) \Bigr)^{\tfrac12} \biggr) \label{functional_2}
\end{align}
with $\lambda >0$ being a regularization parameter and $\varphi\colon \mathbb R_{\ge 0} \rightarrow \mathbb R_{\ge 0}$.
As before we set $\varphi(t) \coloneqq \varphi(-t)$ for $t<0$ and consider $\varphi$ as a function defined on the whole real axis.
For $\varphi(t) \coloneqq \lvert t\rvert$, the second sum in the regularizer of $J_\nu$ is just $\bigl\lVert\bigl(d(u_i,u_j) \bigr)_{j \in {\mathcal N}(i)^+} \bigr\rVert_\nu$, $\nu \in \{ 1,2 \}$.
Then $J_1$ resembles the setting in~\cite{WDS14} and is related to the anisotropic ROF functional  \eqref{TV_aniso}
and $J_2$ gives the approach~\cite{LSKC13}  related to the isotropic case \eqref{TV_iso}.
In this paper we will consider smooth regularization terms, i.e., even, differentiable functions $\varphi$.
We will compute a minimizer of $J_\nu$, $\nu \in \{ 1,2 \}$, by half-quadratic minimization methods.

In the following, we briefly recall the reformulation idea of half-quadratic minimization using the concept of the $c$-transform
and apply it to \eqref{functional_1} and \eqref{functional_2}.
The $c$-transform of functions defined on metric spaces is used 
in connection with optimal transport problems see, e.g., \cite[p. 86f]{Villani03} and seems also to be an appropriate
approach here.
Given a function $c: \mathbb R \times \mathbb R \rightarrow \mathbb R$, 
the $c$-transform of a function $\varphi: \mathbb R \rightarrow \mathbb R$
is defined by 
\[
\varphi^c (s) \coloneqq \inf_{t \in \mathbb R} \bigl\{ c(t,s) - \varphi(t) \bigr\}.
\]
By this definition we see that
$
\varphi(t) + \varphi^c (s) \le c(t,s)
$.
For $c(t,s) \coloneqq -st$ we have $\varphi^c = -(-\varphi)^*$ with the Fenchel transform 
defined for $h:\mathbb R \rightarrow \mathbb R$ as
\[
h^* (s) \coloneqq \sup_{t \in \mathbb R} \bigl\{ ts - h(t) \bigr\}.
\]
Recall that 
$
h^{**}  = h
$ 
if and only if $h$ is lower semi-continuous (lsc) and convex.
In the {\it multiplicative} and {\it additive} half-quadratic methods we use the functions
\begin{align} \label{mult_c}
c(t,s) &\coloneqq t^2 s, \qquad \qquad \qquad \qquad \qquad \; \; \; {\rm (multiplicative)} \\
c(t,s) &\coloneqq \frac12 \Bigl(\sqrt{a}\ t - \frac{1}{\sqrt{a}} s \Bigr)^2,  \quad a>0, \quad{\rm (additive)} \label{add_c}
\end{align}
respectively. The multiplicative setting was introduced in~\cite{GR92} and the
additive one in~\cite{GY95}. The quadratic cost function \(c(t,s)\) in the
additive setting is also handled in optimal transport topics. Note that
the cost function \(c(t,s)\) in the multiplicative case is not bounded from
below in $s$. The following proposition is crucial for the half-quadratic reformulation.
%
\begin{proposition} \label{prop_1}
Let $\varphi\colon \mathbb R \rightarrow \mathbb R_{\ge 0}$ be an even, differentiable function
and let $c\colon \mathbb R \times \mathbb R \rightarrow \mathbb R$ be given by \eqref{mult_c}
and \eqref{add_c}, respectively.
\begin{itemize}
\item[{\rm i)}] If the functions 
\begin{align} \label{phi_mult}
\Phi (t) &\coloneqq \begin{cases}
 -\varphi(\sqrt{t}) &\text{for } t \ge 0,\\
 +\infty &\text{for } t < 0,
\end{cases}
\quad {\rm (multiplicative)}
\\
\Phi (t) &\coloneqq \frac12 a t^2 - \varphi(t), \qquad \qquad \qquad {\rm (additive)} \label{phi_add}
\end{align}
respectively, are convex, then 
$
\varphi = \varphi^{cc},
$
i.e., setting $\psi(s) \coloneqq \varphi^c(s)$ one has
\begin{align} \label{dual_1}
 \psi(s) = \inf_{t \in \mathbb R} \bigl\{c(t,s) - \varphi(t) \bigr\},\\
 \varphi(t) = \inf_{s \in \mathbb R} \bigl\{c(t,s) - \psi(s) \bigr\} \label{dual_2}.
\end{align}
\item[{\rm ii)}] If in addition to the assumption in {\rm i)} we have
\begin{align} \label{existence_mult}
&\lim_{t \rightarrow \infty} \frac{\varphi(t)}{t^2} \rightarrow 0,  
 \quad \; {\rm (multiplicative)}\\
&\lim_{t \rightarrow \infty} \frac{\varphi(t)}{t^2} < \frac12 a, 
\quad  {\rm (additive)} \label{existence_add}
\end{align}
respectively, and in the multiplicative case also
$\varphi'(t) \ge 0$ for $t \ge 0$ and $\varphi''(0+) \coloneqq \lim_{t \rightarrow 0+} \frac{\varphi'(t)}{t}$ exists,
then the infimum in \eqref{dual_1} and \eqref{dual_2}
is attained for $(t,s) = \bigl(t,s(t)\bigr)$ with
\begin{align} \label{mini_1}
s(t) &\coloneqq \begin{cases}
  \frac{\varphi'(t)}{2t} &\text{for } t > 0,\\[1ex]
  \frac{\varphi''(0+)}{2}&\text{for } t = 0,
\end{cases}
\quad {\rm (multiplicative)}
\\
  s(t) &\coloneqq at - \varphi'(t), \quad  
	\quad \qquad  \quad \quad  {\rm (additive)} \label{mini_2}
\end{align}
respectively, and for these pairs we have $\varphi(t) + \psi(s) = c(t,s)$.
The choice is unique except for the multiplicative case and $t=0$, where any $s$ larger than $\frac{\varphi''(0+)}{2}$
is also a solution.
\item[{\rm iii)}] If in the multiplicative case in addition $\varphi'(t) > 0$ for $t > 0$ and
$\varphi''(0+) > 0$, then $s \in (0,\frac{\varphi''(0+)}{2}]$.
\end{itemize}
\end{proposition}
%
Note that in the multiplicative case $\psi(s) = -\infty$ for $s < 0$, so that we can restrict our attention in the infimum in \eqref{dual_2}
to $s \ge 0$. 
Further, the assumption $\varphi'(t) \ge 0$, $t \ge 0$, is in particular fulfilled if $\varphi(t)$  is convex and $\varphi''(0+) \ge 0$.
The proof can be given following for example the lines in~\cite{CBAB97,NN05}.
However, since the assumptions in these papers are slightly different
and the $c$-transform notation is not used, we add the proof 
in the appendix~\ref{app:proofs} to make the paper self-contained.
The functions $\varphi$ which fulfill the conditions in Proposition \ref{prop_1}
and which were used in our numerical test are listed in Table \ref{possible_phi}.
Further examples are collected in~\cite{CBAB97,NN05,NC07}.

In the following we assume that $\varphi$ fulfills the assumptions of Proposition \ref{prop_1}.
Now the idea is to replace $\varphi$ in  \eqref{functional_1}, resp., \eqref{functional_2} 
by the expression in \eqref{dual_2} and to consider
$\min_u J_\nu(u) = \min_u \min_v {\mathcal J}_\nu (u,v)$ with 
\begin{align}
 {\mathcal J}_1(u,v) &\coloneqq \frac12 \sum_{i \in {\mathcal V} } d^2(u_i,f_i) + \lambda  \sum_{i \in {\mathcal G}} \sum_{j \in {\mathcal N}(i)^+ } 
  \Bigl( c\bigl(d(u_i,u_j),v_{i,j} \bigr) - \psi(v_{i,j}) \Bigr) ,
 \label{efunctional_1} \\
 {\mathcal J}_2(u,v) &\coloneqq \frac12 \sum_{i \in {\mathcal V} } d^2(u_i,f_i) + \lambda  \sum_{i \in {\mathcal G}} 
 \Biggl(c\biggl( \Bigl( \sum_{j \in {\mathcal N}(i)^+ } d^2(u_i,u_j) \Bigr)^{\tfrac12},v_{i} \biggr) - \psi(v_{i}) \Biggr),
 \label{efunctional_2}
\end{align}
where we have used the notation $v \coloneqq (v_{i,j})_{i,j\in {\mathcal G}}$ in the anisotropic case and 
$v \coloneqq (v_{i})_{i\in {\mathcal G}}$ in the isotropic case.
By Proposition \ref{prop_1} i), minimizing ${\mathcal J}_\nu$ over $u$ and $v$ gives the same solutions for $u$ as just minimizing
$J_\nu$ over $u$. More precisely we give the following remark.

%
\begin{remark} \label{rem:critical}
Let us abbreviate
${\rm d}_{i,j} \coloneqq d(u_i,u_j)$, ${\rm d}_u \coloneqq ({\rm d}_{i,j})_{i,j\in {\mathcal G}}$
in the anisotropic case 
and
${\rm d}_i \coloneqq \bigl(\sum\limits_{j \in {\mathcal N}(i)^+ } d^2(u_i,u_j) \bigr)^{\tfrac12}$, ${\rm d}_u \coloneqq ({\rm d}_{i})_{i\in {\mathcal G}}$
in the isotropic case. 
If $\hat u$ is a minimizer of $J_\nu$, $\nu = 1,2$, then
$\bigl(\hat u, s({\rm d}_{\hat u}) \bigr)$
is a minimizer of ${\mathcal J}_{\nu}$, $\nu = 1,2$ and conversely.
In particular, if 
\[
\hat u = \argmin_u {\mathcal J}_\nu \bigl(u, s(\hat u)\bigr)
\]
holds true, then
$\hat u$ is a minimizer of  $J_\nu$.
\end{remark}
%

Now we can apply an alternating minimization over $v_{i,j}\in \mathbb R$, resp., $v_i\in \mathbb R$ and $u \in {\mathcal M}^{n\times m}\eqqcolon \mathrm{M}$:
\begin{align}
  v^{(k+1)} &\in  \argmin_v {\mathcal J}_\nu(u^{(k)},v) \label{step_2},\\
  u^{(k+1)} &\in \argmin_u {\mathcal J}_\nu (u,v^{(k+1)}) \label{step_1}.
\end{align}
Clearly, under the assumptions of Proposition \ref{prop_1} we have
\begin{equation} \label{rel_J_J}
 {\mathcal J}_\nu\bigl(u^{(k)},v^{(k+1)}\bigr) = J_\nu \bigl(u^{(k)}\bigr).
\end{equation}
We want to work with the differentiable function $d^2$ in the second iteration \eqref{step_1}.
Since the additive reformulation leads to a non-differentiable function $d^2+sd$,
we restrict our attention in the following to the multiplicative case. 

\paragraph{Minimization with respect to $v$.}
The minimization over $v$ in \eqref{step_2} can be done separately for the
$v_{i,j}$ or $v_i$.
By Proposition \ref{prop_1} a minimizer is given by 
\[
	v^{(k+1)} \coloneqq s \bigl({\rm d}_{u^{(k)}} \bigr),
\]
where $s$ is defined as in \eqref{mini_1}.
Note that only for $d\bigl(u_i^{(k)}, u_j^{(k)}\bigr) = 0$ in the anisotropic case 
and
$\sum\limits_{j \in {\mathcal N}(i)^+ } \Bigl(d^2\bigl(u_i^{(k)},u_j^{(k)}\bigr) \Bigr)^{\tfrac12} = 0$ 
in the isotropic case a larger value also could be taken as a minimizer.

If $\varphi$ fulfills the assumptions of Proposition \ref{prop_1} iii),
then $v^{(k+1)} \in \bigl(0,\varphi''(0+)/2\bigr]$.
\paragraph{Minimization with respect to $u$.} 
The minimization over $u$ in \eqref{step_1} is equivalent to finding the minimizer of
\begin{align}
  {\mathcal J}_{1,v^{(k)}} (u)&\coloneqq \frac12 \sum_{i \in {\mathcal V} } d^2(u_i,f_i) 
	+ \lambda  \sum_{i \in {\mathcal G}} \sum_{j \in {\mathcal N}(i)^+ } 
 d^2(u_i,u_j) v_{i,j}^{(k)}  ,
 \label{efunctional_1_u} \\
 {\mathcal J}_{2,v^{(k)}} (u) &\coloneqq \frac12 \sum_{i \in {\mathcal V} } d^2(u_i,f_i) + \lambda  \sum_{i \in {\mathcal G}} 
 \Bigl( \sum_{j \in {\mathcal N}(i)^+ } d^2(u_i,u_j) \Bigr) v_{i}^{(k)},
 \label{efunctional_2_u}
\end{align}
respectively.
We can apply, e.g., a gradient descent or a Riemann-Newton method, see~\cite{AMS08}.
Both methods are described in the following for our setting and were implemented. 

We need the following notation.
Let 
$T_x{\mathcal M}$ denote the tangential space of ${\mathcal M}$ at $x \in {\mathcal M}$ and
$\langle\cdot,\cdot\rangle_x\colon 	T_x{\mathcal M}\times T_x\mathcal M \rightarrow \mathbb R$ 
the Riemannian metric with induced norm \(\lVert\cdot\rVert_x\).
Let \(\gamma_{x,\xi}(t)\), \(x\in\mathcal M\), \(\xi\in T_x\mathcal M\) be
the minimal geodesic starting from \(\gamma_{x,\xi}(0) = x\) with
\(\dot\gamma_{x,\xi} (0) = \xi\). 
Then the exponential map \(\exp_x\colon T_x\mathcal M \rightarrow \mathcal M\) is given by
\(\exp_x\xi = \gamma_{x,\xi}(1)\).  
The inverse exponential map denoted by
\(\log_x = \exp_x^{-1}\colon \mathcal M \to T_x\mathcal M\)
is locally well-defined. 
For the manifolds used in our numerical examples, namely
the sphere \(\mathbb S^n\), \(n\in\mathbb N\), the $\operatorname{SO}(3)$ and the manifold \(\mathcal P(r)\), \(r\in\mathbb N\), of symmetric positive definite \(r\times r\) matrices,
the specific maps are given in the Appendix~\ref{app:B}.
Finally, for~$F\colon \mathcal M\rightarrow \mathbb R$, let 
\[
\grad  F (x) \in T_x \mathcal M
\quad{\rm and} \quad
\mathrm{Hess}_{F} (x): T_x \mathcal M \rightarrow T_x \mathcal M
\] 
be the Riemannian gradient and the Hessian of~$F\colon \mathcal M\rightarrow \mathbb R$ 
at $x \in \mathcal M$, respectively.
For $d (\cdot,y): \mathcal M \rightarrow \mathbb R_{\ge 0}$ one has
\[
\grad d^2 (x,y) = - 2 \log_x y.
\]
Considering the image $u \in {\rm M}$,
we abbreviate the gradient and the Hessian of a function $F:{\rm M} \rightarrow \mathbb R$  
defined on the product manifold ${\rm M}$ also by $\grad F$ and $\mathrm{Hess}_{F}$, resp., 
since its use becomes clear from the context.
\paragraph{Gradient descent method.} 
The gradient descent method computes, starting with $\tilde u^{(0)} := u^{(k)}$, iteratively 
\begin{equation} \label{grad_desc}
\tilde u^{(r+1)} = \exp_{\tilde u^{(r)}} \Bigl(t_r \grad {\mathcal J}_{\nu,v^{(k+1)}}(u)|_{u=\tilde u^{(r)}}\Bigr), \quad \nu \in \{1,2 \},
\end{equation}
with appropriate step sizes  $t_r >0$.
The gradient $\grad {\mathcal J}_{\nu,v^{(k+1)}}$ is given by
\begin{align} \label{grad_aniso}
\left(\grad {\mathcal J}_{1,v^{(k+1)}} (u)\right)_{i} &=  -1_{\mathcal{V}(i)}\log_{u_{i}} f_i - 2\lambda 
\sum_{j \in {\mathcal N}(i)} v_{i,j}^{(k+1)}\log_{u_i} u_j,\\
\left(\grad {\mathcal J}_{2,v^{(k+1)}} (u)\right)_{i} &=  -1_{\mathcal{V}(i)}\log_{u_{i}} \! \!f_i \!- \!2\lambda\Bigl(\!  v_{i}^{(k+1)}\!  \!  \! \!
\sum_{j \in {\mathcal N}(i)^+} \!\! \log_{u_i} \! \! u_j \!- \!  \!  \!\! \!\sum_{j \in {\mathcal N}(i)^-}\!   \! v_{j}^{(k+1)} \log_{u_i} \! \! u_j\!  \Bigr)\label{grad_iso}
\end{align}
where $i \in {\mathcal G}$, 
${\mathcal N}(i)^- \coloneqq \bigl\{ (i_1-1,i_2),(i_1,i_2-1)\bigr\}$ 
and 
${\mathcal N}(i) \coloneqq {\mathcal N}(i)^+ \cup {\mathcal N}(i)^-$.
\paragraph{Riemann--Newton method.}
Alternatively we can use a Riemann--Newton method to compute a minimizer.
Finding a descent direction $\eta_r\in T_{\tilde u^{(r)}} {\rm M}$ with Newton's method 
is done for \(\nu\in\{1,2\}\) by solving the system of equations
\begin{align} \label{newton_eq}
\mathrm{Hess}_{\mathcal{J}_{\nu,v^{(k+1)}}}\bigl(\tilde u^{(r)}\bigr)(\eta_r) &= -\grad {\mathcal J}_{\nu,v^{(k+1)}}\bigl(\tilde u^{(r)}\bigr).
\end{align}
Then we update, starting with $\tilde u^{(0)} := u^{(k)}$, iteratively 
\[
\tilde u^{(r+1)} = 
\begin{cases}
	\exp_{\tilde u^{(r)}}\eta_r &\text{if } \bigl\langle\eta_r,\grad \mathcal J_{\nu,v^{(k+1)}}(\tilde u^{(r)})\bigr\rangle_{\tilde u^{(r)}} < 0,\\
 \exp_{\tilde u^{(r)}}\Bigl(-\grad \mathcal J_{\nu,v^{(k+1)}}\bigl(\tilde u^{(r)}\bigr)\Bigr)&\text{otherwise}.
\end{cases}
\]
The whole half-quadratic minimization method for our problem is given in Algorithm \ref{alg:half_quad}. 
\\

\begin{algorithm}[tbp]
	\caption[]{\label{alg:half_quad} Image Restoration by Half-Quadratic Minimization (multiplicative)}
	\begin{algorithmic}
		\STATE \textbf{Input:}  ${\mathcal V}$, corrupted image $f \in \mathcal{M}^{\# {\mathcal V}}\subseteq\mathrm{M}$, $\lambda$, $\varphi$
		\STATE \textbf{Output:} Restored image $u \in {\rm M}$ 
		\STATE \textbf{Initialize} $u^{(0)}$
		\REPEAT
		\STATE $k\leftarrow k+1$;
				\STATE $v^{(k+1)} = s\bigl(d^{(k)}\bigr)$;
		\STATE Compute 
		\STATE $u^{(k+1)} \coloneqq \exp_{u^{(R)}}\eta_R$ 
		\STATE by $R$ steps of a gradient descent method \eqref{grad_aniso}, resp. \eqref{grad_iso} 
		\STATE or by Newton's approach \eqref{newton_eq};
		\UNTIL{stopping criterion is reached;}
		\end{algorithmic}
\end{algorithm}

In~\cite{NC07} it was shown that the multiplicative half-quadratic minimization is equivalent to the quasi-Newton descent method, 
and therefore we expect that its performance is better than the simple gradient descent method.
Let us comment on this for the manifold-valued setting. 
\begin{remark} (Relation of  half-quadratic minimization to gradient descent 
and (quasi) Newton methods) \label{rem:methods}\\
We restrict our attention to the case $\nu =1$. Similar considerations can be done for $\nu = 2$.
The gradient of the initial functional $J_1$ in \eqref{functional_1} is given for $u_i \not = u_j$ by
	\begin{flalign*}
	\bigl(\grad J_1(u)\bigr)_i &= -1_{\mathcal{V}(i)} \log_{u_i}f_i-\lambda \sum_{j \in\mathcal{N}(i)}\varphi^\prime\bigl(d(u_i,u_j)\bigr)\frac{\log_{u_i}u_j}{\|\log_{u_i}u_j\|_{u_i}}\\
	&=-1_{\mathcal{V}(i)} \log_{u_i}f_i-2\lambda \sum_{j \in\mathcal{N}(i)}\frac{\varphi^\prime\bigl(d(u_i,u_j)\bigr)}{2d(u_i,u_j)}\log_{u_i}u_j
	\end{flalign*}
	which by \eqref{mini_1} can be rewritten as
	\begin{equation}
	\bigl(\grad J_1(u)\bigr)_i =-1_{\mathcal{V}(i)} \log_{u_i}f_i-2\lambda \sum_{j \in\mathcal{N}(i)}s\bigl(d(u_i,u_j)\bigr)\log_{u_i}u_j.\label{linear_grad}
	\end{equation}
Hence a gradient descent algorithm applied to the initial functional $J_1$ coincides with the half-quadratic method 
if we perform only one step in the gradient descent method \eqref{grad_desc} to obtain an update of $u$.
More than one gradient descent step in \eqref{grad_desc} leads to a linearized gradient descent of $J_1$. 
If we perform a Riemann--Newton step \eqref{newton_eq} to update $u$
we have a quasi-Newton method for $J_1$. 
Note that $v^{(k+1)} = s({\rm d}_{u^{(k)}})$ is fixed in the half-quadratic update step of $u$
which is not the case in \eqref{linear_grad}. This simplifies the computation of the Hessian in the half-quadratic approach.
\end{remark}

\newpage
\section{Convergence in Hadamard Manifolds} \label{sec:hada}
We start with a general remark.
\begin{remark} \label{rem_u}
Assume that $\varphi\colon \mathbb R \rightarrow \mathbb R_{\ge 0}$ fulfills the assumptions
of Proposition \ref{prop_1} iii).
Let $\bigl\{(u^{(k)},v^{(k)})\bigr\}_k$ be the sequence produced by Algorithm~\ref{alg:half_quad}.
Then we know that $v^{(k)} \in \bigl(0, \varphi''(0+)/2\bigr]^p$, where $p=2\#\mathcal{G}$, if $\nu=1$ (anisotropic case), and $p=\#\mathcal{G}$, if $\nu=2$ (isotropic case).
By construction we have for the iterates produced by Algorithm~\ref{alg:half_quad} that
\begin{equation} \label{abstieg}
{\mathcal J}_\nu \bigl(u^{(k)},v^{(k)}\bigl) \ge {\mathcal J}_\nu \bigl(u^{(k)},v^{(k+1)}\bigr) \ge {\mathcal J}_\nu \bigl(u^{(k+1)},v^{(k+1)}\bigr)
\end{equation}
so that the sequence $\bigl\{ {\mathcal J}_\nu \bigl(u^{(k)},v^{(k)}\bigr)\bigr\}_{k \in \mathbb N}$ is monotonically decreasing.
By \eqref{dual_2} and since $\varphi$ is nonnegative, the function
${\mathcal J}_\nu$ is bounded from below by zero and the sequence $\bigl\{ {\mathcal J}_\nu \bigl(u^{(k)},v^{(k)}\bigr) \bigr\}_k$ 
converges to some $b_\nu$. 
This holds also true for $J_\nu\bigl(u^{(k)}\bigr)$ by \eqref{rel_J_J}.
If ${\mathcal M}$ is compact as in the case of spheres or $\operatorname{SO}(3)$, then $\bigl\{u^{(k)}\bigr\}_k$ is clearly bounded.
If ${\mathcal M}$ is an Hadamard space as defined in the next subsection and $\varphi$ is coercive, then $\bigl\{u^{(k)}\bigr\}_k$
is bounded since $J_\nu$ is by Proposition \ref{prop_2} coercive. 
In these cases $\bigl\{\bigl(u^{(k)},v^{(k)}\bigr)\bigr\}_k$ is also bounded and
therefore there exists a subsequence $\bigl\{\bigr(u^{(k_j)},v^{(k_j)}\bigr)\bigr\}_j$ which converges to a point $(\bar u, \bar v) \in {\rm M} \times \bigl[0,\varphi''(0)/2\bigr]^p$.
\end{remark}

For Hadamard spaces many results on the convergence of algorithms carry directly over from the Hilbert space setting.
This is in particular true for the half-quadratic minimization algorithm.
In this section we summarize these results for convex functions $\varphi$ and data in Hadamard spaces.

We start by recalling some basic facts.
A curve $\gamma\colon [0,1] \rightarrow X$ in a metric space $(X,d)$ is called a geodesic if for all $t_1,t_2 \in [0,1]$
the relation
\[
	d \bigl(\gamma(t_1),\gamma(t_2) \bigr) = \rvert t_1 - t_2\lvert d \bigl(\gamma(0),\gamma(1) \bigr)
\]
holds true. A function $h: X \rightarrow \mathbb R$ is called convex if $h \circ \gamma$ is convex for each geodesic
$\gamma\colon[0,1] \rightarrow X$, i.e., if for all $t \in [0,1]$ we have
\[
h\bigl( \gamma(t) \bigr) \le t h\bigl( \gamma (0) \bigr) + (1-t)h \bigl(\gamma(1)\bigr)
\]
and strictly convex if we have a strict inequality for all $0 < t <1$. An Hadamard space is a complete metric space $({\mathcal H},d)$ with the property that any two points $x,y$ 
are connected by a geodesic and the following condition holds true
\begin{equation}\label{eq:reshet}
d(x,v)^2 + d(y,w)^2 \le d(x,w)^2 + d(y,v)^2 + 2 d(x,y)d(v,w),
\end{equation}
for any $x,y,v,w \in X.$ Inequality \eqref{eq:reshet} implies that Hadamard spaces have nonpositive curvature~\cite{Alex51,Re68} and Hadamard spaces are 
thus a natural generalization of complete simply connected Riemannian manifolds of nonpositive sectional curvature,
the so-called Hadamard manifolds. For more details, the reader is referred to~\cite{bacak14,Jost97}.
Unfortunately, the spheres and the rotation group are not Hadamard manifolds, while the symmetric positive definite matrices
have this nice property.
The following facts can be shown similarly as in $\mathbb R^d$, see \cite[Lemma 2.2.9]{bacak14},~\cite{sturm03}.

\begin{lemma}\label{prop_3}
Let $({\mathcal H}, d)$ be an Hadamard space and $F\colon {\mathcal H} \rightarrow \mathbb R \cup \{+\infty\}$ 
be a convex lsc function which is coercive, i.e., satisfies $F(x) \rightarrow +\infty$ whenever
$d(x,x_0) \rightarrow +\infty$ for some $x_0 \in {\mathcal H}$. 
Then $F$ has a minimizer. If $F$ is convex, then any critical point is a  global minimizer.
If $F$ is coercive and strictly convex, then the minimizer is unique.
\end{lemma}

In an Hadamard space $({\mathcal H},d)$ we have that 
\begin{itemize}
\item[(D1)]  $d\colon {\mathcal H} \times {\mathcal H} \rightarrow \mathbb R_{\ge 0}$  and 
$d^2\colon {\mathcal H} \times {\mathcal H} \rightarrow \mathbb R_{\ge 0}$ are convex, and
\item[(D2)] $d^2(\cdot,y)\colon {\mathcal H} \rightarrow \mathbb R_{\ge 0}$ is strictly convex.
\end{itemize}
Then we obtain the following proposition whose simple proof is added 
for convenience in Appendix~\ref{app:proofs}.
%
\begin{proposition}\label{prop_2} 
Let $({\mathcal M},d) = ({\mathcal H},d)$ be an Hadamard manifold.
\begin{itemize}
 \item[{\rm i)}] 
 Let $\varphi\colon \mathbb R_{\ge 0} \rightarrow \mathbb R_{\ge 0}$ and  
 in the case ${\mathcal V} \not = {\mathcal G}$, further assume that  $\varphi$ is coercive.
Then the functions ${J}_\nu$, $\nu = 1,2$, in \eqref{functional_1} and \eqref{functional_2}
are coercive so that they have a minimizer.
 \item[{\rm ii)}]
 If in addition $\varphi$ is increasing and convex,
then the functions $J_\nu$, $\nu = 1,2$,
are convex. If in addition ${\mathcal V} = {\mathcal G}$ or $\varphi$ is strictly convex,
then the functions $J_\nu$, $\nu = 1,2$, are strictly convex and have  unique minimizers.
 \end{itemize}
\end{proposition}

Under the assumptions of Proposition \ref{prop_1} iii),
we have in our algorithm that $v^{(k)} >0$.
Then, we see similarly as in the proof of Proposition \ref{prop_2} that the functionals 
$J_{\nu,v^{(k)}}$, $\nu = 1,2$,
are coercive and strictly convex. Thus the minimizer $u^{(k)}$ exists and is unique.

\begin{theorem} \label{convergence}
Let $({\mathcal M},d) = ({\mathcal H},d)$ be an Hadamard manifold.
Let $\varphi\colon \mathbb R \rightarrow \mathbb R_{\ge 0}$ be an even, 
continuously differentiable, convex function which fulfills
\begin{itemize}
\item[{\rm i)}] $\varphi(\sqrt{t})$, $t>0$ is concave,
\item[{\rm ii)}] $\lim_{t \rightarrow \infty} \frac{\varphi(t)}{t^2} \rightarrow 0$,
\item[{\rm iii)}] $\varphi''(0+) > 0$.
\end{itemize}
In the case ${\mathcal V} \not = {\mathcal G}$ we further assume that $\varphi$ is strictly convex.
Then the sequence $\bigl\{ u^{(k)} \bigr\}_{k \in \mathbb N}$ generated by Algorithm \ref{alg:half_quad} 
converges to the minimizer of $J_\nu$, $\nu = 1,2$.
\end{theorem}

The proof which follows standard arguments is given in the appendix~\ref{app:proofs}.
Note that the assumptions of the theorem are fulfilled for the first two functions in Table \ref{possible_phi}.

\begin{remark} \label{assumptions}
The assumptions in the theorem include those in Proposition \ref{prop_1}.
Note that the convexity of $\varphi$ and $\varphi''(0+) > 0$ imply that $\varphi'(t) > 0$ for all $t>0$.
Additionally the continuity of $\varphi'$ is required to make the function $s$ in \eqref{mini_1}
continuous and the (strict) convexity of $\varphi$ to make the objective function strictly convex.
Since $\varphi$ is convex, its derivative is increasing.
Together with $\varphi''(0+) > 0$ this implies that $\varphi'(t) > 0$ for all $t>0$
so that $\varphi$ is increasing for $t>0$ and coercive.
\end{remark}

\section{Numerical Examples} \label{sec:num}
%
In this section we demonstrate the performance of Algorithm \ref{alg:half_quad} 
for the functions $\varphi$ from Table \ref{possible_phi}. 
These functions are known for their edge-preserving properties.
Note that $\varphi_1$ was used in the ``lagged diffusivity fixed point iteration'' \cite{VO96}
for real-valued images and in the iteratively re-weighted least squares method \cite{GS14}
for $\mathbb S^2$-valued and ${\cal P}(3)$-valued images.
The function $\varphi_2$ is a Moreau envelope of the absolute value function, also known as the Huber function. 
Both functions are convex.
The non-convex function $\varphi_3$ was used for edge-preserving restoration of real-valued images in \cite{NC07,CBAB97}.
\begin{table}
	\setlength{\tabcolsep}{1em}
	\centering
	\begin{tabular}{lll} 
		\toprule
		&$\varphi(t)$ & $s(t)$  \\
		\midrule
		$\varphi_1(t)$ & $\sqrt{t^2+\varepsilon^2}$ & $\frac{1}{\sqrt{t^2+\varepsilon^2}}$ \\\addlinespace[1.5ex]
		$\varphi_2(t)$ & $\begin{cases}\frac{1}{2}t^2\quad &t<\varepsilon,\\
		\varepsilon\lvert t\rvert-\frac{1}{2}\varepsilon^2\quad& t\le\varepsilon \end{cases}$ & $\begin{cases}1\quad &t<\varepsilon,\\
		\frac{\varepsilon}{\lvert t\rvert}\quad& t\le\varepsilon \end{cases}$\\\addlinespace[1.5ex]
		$\varphi_3(t)$ & $1-\exp(-\varepsilon^2 t^2)$ & $\varepsilon^2 \exp(-\varepsilon^2 t^2)$\\
		\bottomrule
	\end{tabular}
	\caption{Functions $\varphi$ fulfilling the assumptions of Proposition \ref{prop_1}.
			}\label{possible_phi}
\end{table}
%

%

Unless stated otherwise we use the anisotropic approach and Newton's method in our implementations.
Although neither the spheres nor the rotation group are Hadamard
manifolds we have observed convergence in all our numerical examples.
This may be due to the fact that neighboring image pixels have values which are close enough on the manifold.

The algorithms where implemented in \textsc{MatLab} Version $14$b. 
The computations were performed on a Dell with $8$ GB of RAM and an Intel Core i7, 2.93 GHz, on Ubuntu 14.04 LTS.
\subsection{$\mathbb S^1$-valued data}
We start with the one-dimensional signal in Fig.~\ref{signal} to show how the different functions $\varphi$ 
from Table \ref{possible_phi} perform and how the parameter $\varepsilon$ influences the results. 
The original signal in Fig.~\ref{signal1} was obtained from $f(x) = 8\pi x^2$ by sampling with size $0.01$
and unwrapping modulo $2\pi$ such that the data are represented in $[-\pi, \pi)$.
Then wrapped Gaussian noise with standard deviation $\sigma = 0.3$ was added.
Using $\varphi_1$ to restore the signal gives relatively larges error 
in Fig.~\ref{signal2}. The Huber function $\varphi_2$ and the exponential function $\varphi_3$ 
show better results in Fig.~\ref{signal3} and \ref{signal4}.
The regularization parameter $\lambda$ was adapted to get the best error
\begin{equation}\label{err}
\mathrm{err} \coloneqq \frac{1}{N}\sum_{i=0}^{N-1} d\left(f_i,u_i\right),
\end{equation}
where $N = 101$.
Making $\varepsilon$ in the Huber function larger leads to the smoother result in Fig.\ref{signal5}
which approximates the original signal only well at the beginning of the signal. 
In Fig.~\ref{signal6} we choose a larger $\varepsilon$ in the function $\varphi_3$ with the
effect 
that edges of smaller height are smoothed
and we have a staircasing effect for nearly equally high ascents.
\begin{figure}[tbp]
	\centering
	\begin{subfigure}{0.32\textwidth}
		\includegraphics{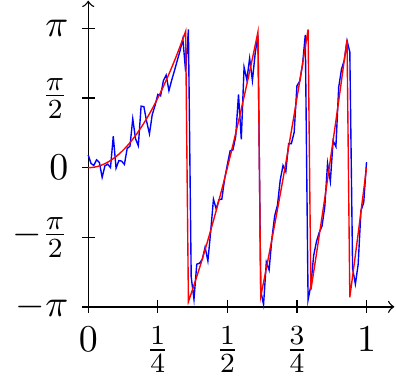}
		\caption{Original and noisy signal}\label{signal1}
	\end{subfigure}
	\begin{subfigure}{0.32\textwidth}
		\includegraphics{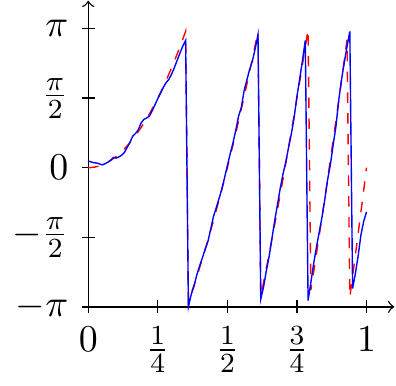}
		\caption{$\varphi_1$}\label{signal2}
	\end{subfigure}
	\begin{subfigure}{0.32\textwidth}
		\includegraphics{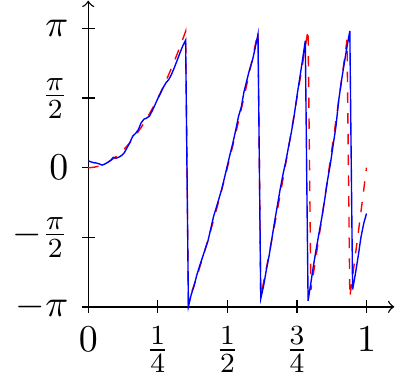}
		\caption{$\varphi_2$}\label{signal3}
	\end{subfigure}
	\begin{subfigure}{0.32\textwidth}
		\includegraphics{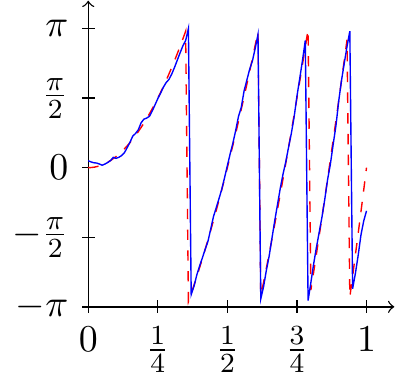}
		\caption{$\varphi_3$}\label{signal4}
	\end{subfigure}
	\begin{subfigure}{0.32\textwidth}
		\includegraphics{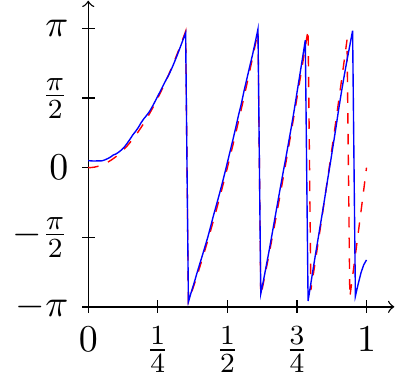}
		\caption{$\varphi_2$, too smooth}\label{signal5}
	\end{subfigure}	
	\begin{subfigure}{0.32\textwidth}
		\includegraphics{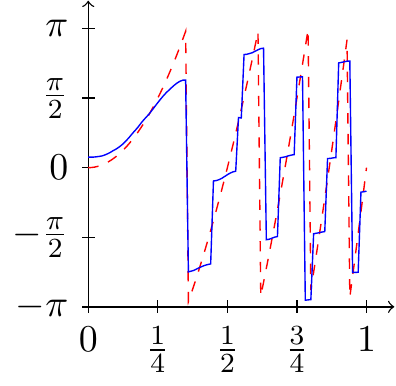}
		\caption{$\varphi_3$, staircasing}\label{signal6}
	\end{subfigure}
	\caption{
	Restoration of a noisy cyclic signal by half-quadratic minimization
	with various functions $\varphi$.
	\subref{signal1}  original (red) and noisy (blue) signals. 
	 Restored signal (blue) using 
	 \subref{signal2} $\varphi_1,\ \varepsilon = 6\times10^{-1},\ \lambda=3.4$, $\mathrm{err}=0.1007$, 
	 \subref{signal3} $\varphi_2,\ \varepsilon = 5\times10^{-1},\ \lambda=5.2$, $\mathrm{err}=0.1007$, 
	 \subref{signal4} $\varphi_3,\ \varepsilon =\frac{1}{\sqrt{2}},\ \lambda=10$,  $\mathrm{err}=0.1001$,
	 \subref{signal5} $\varphi_2,\ \varepsilon = 1,\ \lambda=20$, $\mathrm{err}=0.1733$, 	
	 \subref{signal6} $\varphi_3,\ \varepsilon = \sqrt{5},\ \lambda=10$,  $\mathrm{err}=0.3756$.}\label{signal}
\end{figure}
\\

Next we want to demonstrate the difference between the anisotropic \eqref{efunctional_1} and isotropic \eqref{efunctional_2} half-quadratic minimization methods. 
To this end, the function $\operatorname{atan2}(x,y)$ was sampled over a regular grid $\left[-\frac12,\frac12\right]^2$ with grid size $\frac{1}{128}$, 
resulting in Fig.~\ref{subfig:circular:orig}. 
Then we corrupt the image by removing a circular region from the center as shown in 
Fig.~\ref{subfig:circular:mask}.
Using the anisotropic functional leads to Fig.~\ref{subfig:circular:aniso}, 
where we observe artifacts in vertical and horizontal directions.
The image produced by applying the isotropic functional in Fig.~\ref{subfig:circular:iso}
does not have this problem.
This effect is also illustrated by the error plots in Figs. \ref{subfig:circular:err_aniso} and \ref{subfig:circular:err_iso}. 
\begin{figure}[tbp]\centering
\begin{subfigure}{0.31\textwidth}\centering
	\includegraphics{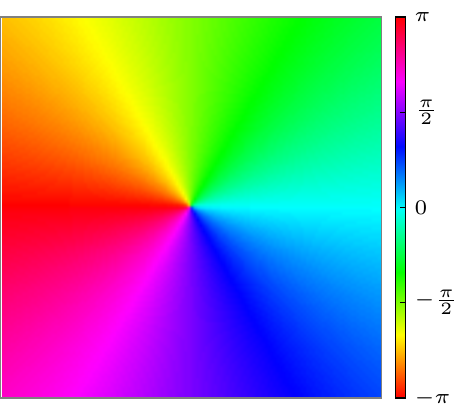}
	\caption{Original image.}\label{subfig:circular:orig}
\end{subfigure}
\begin{subfigure}{0.31\textwidth}\centering
	\includegraphics{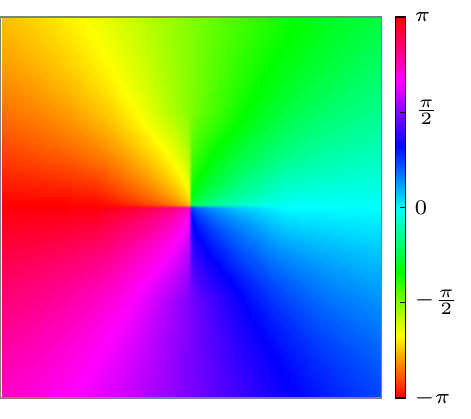}
	\caption{Anisotropic model.}\label{subfig:circular:aniso}
\end{subfigure}
\begin{subfigure}{0.31\textwidth}\centering
	\includegraphics{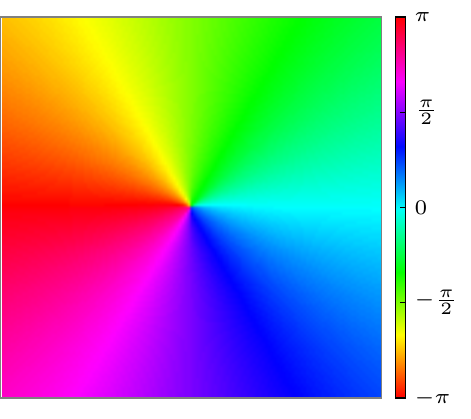}
	\caption{Isotropic model.}\label{subfig:circular:iso}
\end{subfigure}
\begin{subfigure}{0.31\textwidth}\centering
	\includegraphics{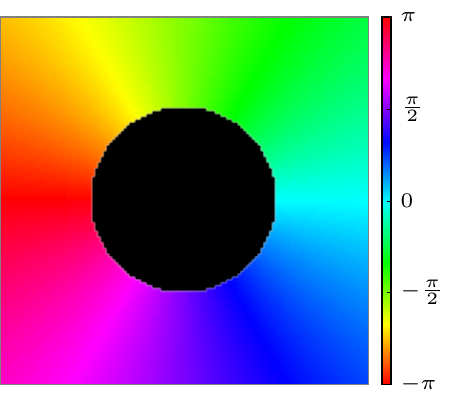}
	\caption{Corrupted image.}\label{subfig:circular:mask}
\end{subfigure}
\begin{subfigure}{0.31\textwidth}\centering
	\includegraphics{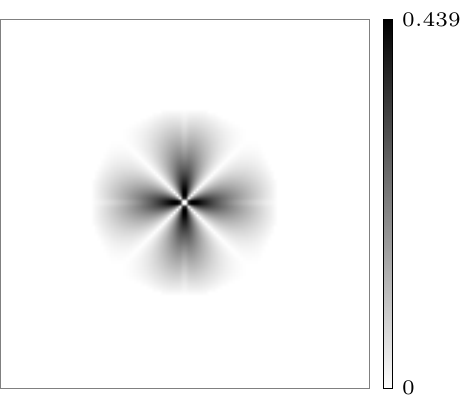}
	\caption{Error of \subref{subfig:circular:aniso}.}\label{subfig:circular:err_aniso}
\end{subfigure}
\begin{subfigure}{0.31\textwidth}\centering
		\includegraphics{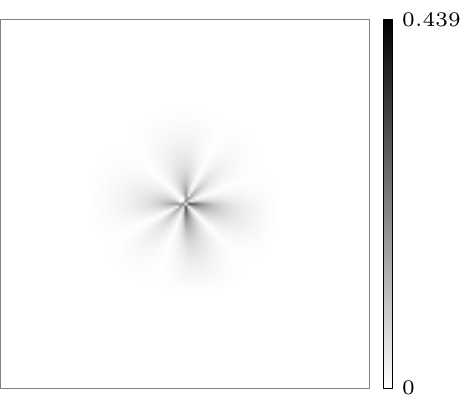}
	\caption{Error of \subref{subfig:circular:iso}.}\label{subfig:circular:err_iso}
\end{subfigure}
\caption{Inpainting of an image with cyclic data using the anisotropic and the isotropic model.
\subref{subfig:circular:orig} original image, 
\subref{subfig:circular:mask} corrupted image.
Restoration with~\subref{subfig:circular:aniso} the anisotropic model and \subref{subfig:circular:iso} the isotropic model
using the function~$\varphi_1$ with $\lambda=0.001$ and $\varepsilon= 10^{-2}$. 
\subref{subfig:circular:err_aniso} and~\subref{subfig:circular:err_iso} error images.}
\label{circular_inpainting}
\end{figure}
%
\subsection{$\mathbb S^2$-valued data}
%
In our first example we denoise color images in the chromaticity and brightness space.
For an RGB image the brightness is given by the real positive numbers $b \coloneqq \bigl(R^2 + G^2 + B^2\bigr)^\frac12$ 
and the chromaticity by the $\Ss^2$-values 
$c \coloneqq (R,G,B)/b$. We want to mention that the first two examples consider problems which have values on the positive octant; only in the last example we look at data covering the whole sphere. We compare half-quadratic minimization with the
different functions $\varphi$ and the TV approach from~\cite{WDS14} in Fig.~\ref{color}. 
We took the image ``Peppers''\footnote{Taken from the USC-SIPI Image
Database, available online at \url{http://sipi.usc.edu/database/database. php?volume=misc&image=15}} 
in Fig.~\ref{subfig:CB:orig} 
and added Gaussian noise with standard deviation $\sigma=0.1$ to all three color channels in the RGB model. 
For denoising in the chromaticity-brightness model we optimized \(\lambda\) with respect to best PSNR for both channels separately 
using a grid search on $\frac{1}{100}\NN$ for \(\varphi_1\) and \(\varphi_3\), and for $\varphi_2$ on $\NN$. 
Furthermore, we optimized $\varepsilon$ first in order of magnitude \(k=10^j\) and refined the search on \(\frac{k}{10}\NN\). 
Both channels are restored using the same function. 
The square-root functional $\varphi_1$ shows smoother transitions at edges while
the Huber function $\varphi_2$ tends more to staircasing.
Using the exponential function $\varphi_3$ yields a worse PSNR and hence does not compete with the previous two functions;
the edges are preserved but within the more constant regions some noise is left. 
The bright spot (see upper magnification) appears also too smooth.
This originates from the too smooth transitions in the brightness 
which are not detected as edges. The TV regularization 
introduces staircasing and it is not able to reduce the noise in the dark area (lower magnification). 
%
\begin{figure}[tbp]
	\centering
	\begin{subfigure}{0.32\textwidth}\centering
		\includegraphics{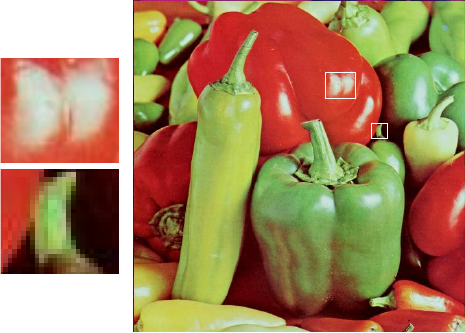}
		\caption{Original image.}\label{subfig:CB:orig}
	\end{subfigure}
	\begin{subfigure}{0.32\textwidth}
		\includegraphics{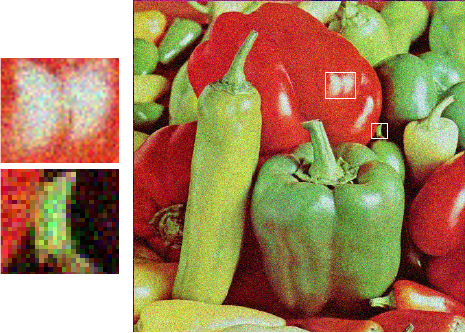}
		\caption{Noisy, PSNR: $20.31$.}\label{subfig:CB:noisy}
	\end{subfigure}
	\begin{subfigure}{0.32\textwidth}\centering
		\includegraphics{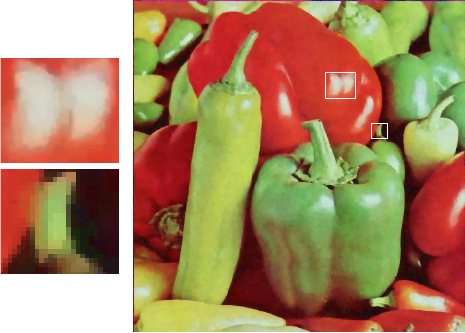}
		\caption{$\varphi_1$, PSNR: $30.29$.}\label{subfig:CB:phi1}
	\end{subfigure}
	\begin{subfigure}{0.32\textwidth}\centering
		\includegraphics{CB-Peppers-phi1}
		\caption{$\varphi_2$, PSNR: $29.68$.}\label{subfig:CB:phi2}
	\end{subfigure}
	\begin{subfigure}{0.32\textwidth}\centering
		\includegraphics{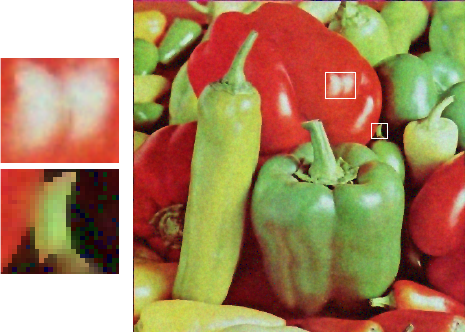}
		\caption{$\varphi_3$, PSNR: $28.95$.}\label{subfig:CB:phi3}
	\end{subfigure}
	\begin{subfigure}{0.32\textwidth}\centering
		\includegraphics{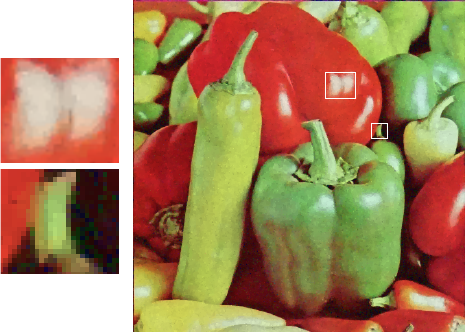}
		\caption{TV, PSNR: $29.32$.}\label{subfig:CB:TV}
	\end{subfigure}
	\caption{Denoising in the chromaticity-brightness space.
	\subref{subfig:CB:orig} Original image ``Peppers''.
	\subref{subfig:CB:noisy} Corrupted image by Gaussian noise on RGB, $\sigma=0.1$. 
	Restored images by 
	\subref{subfig:CB:phi1} 
	$\varphi_1$ using $\varepsilon_c = 10^{-3} ,\ \varepsilon_b = 10^{-2},\ \lambda_c = 0.44,\ \lambda_b=0.08$, 
	\subref{subfig:CB:phi2} 
	$\varphi_2$ using $\varepsilon_c = 10^{-3},\ \varepsilon_b =10^{-3} ,\ \lambda_c = 15,\ \lambda_b=10$, 
	\subref{subfig:CB:phi3} 
	$\varphi_3$ using $\varepsilon_c = 2\sqrt{3},\ \varepsilon_b = \sqrt{23},\ \lambda_c = 0.1,\ \lambda_b=0.03$, and 
	\subref{subfig:CB:TV} 
	TV method in \cite{WDS14}, $\alpha = 0.05$.}\label{color}
\end{figure}

In the second example we use  half-quadratic minimization for colorization in the chromaticity-brightness space. 
We assume that the brightness of the image is known, but 99 percent of the chromaticity information is lost. 
The original image is shown in Fig.~\ref{house_orig} and its corrupted version in Fig.~\ref{house_lost}.
For inpainting the chromaticity we have used a nearest neighbor initialization. 
With the regularizing function $\varphi_1$ we obtain the result depicted in Fig.~\ref{house_hq}. 
We compare this with Fig.~\ref{house_kang} which is obtained by using the chromaticity colorization method in \cite{QKL10}
which we have implemented for comparison.
%
\begin{figure}[tbp]
	\centering
	\begin{subfigure}[t]{0.24\textwidth}
		\includegraphics[width = \textwidth]{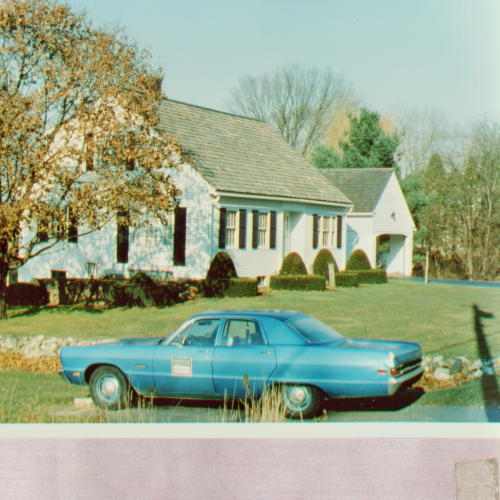}
		\caption{Original image house.}\label{house_orig}
	\end{subfigure}
	\begin{subfigure}[t]{0.24\textwidth}
		\includegraphics[width = \textwidth]{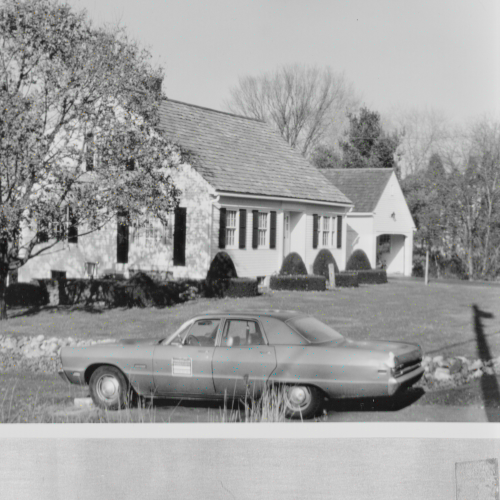}
		\caption{$99\%$ color lost.}\label{house_lost}
	\end{subfigure}
	\begin{subfigure}[t]{0.24\textwidth}
		\includegraphics[width = \textwidth]{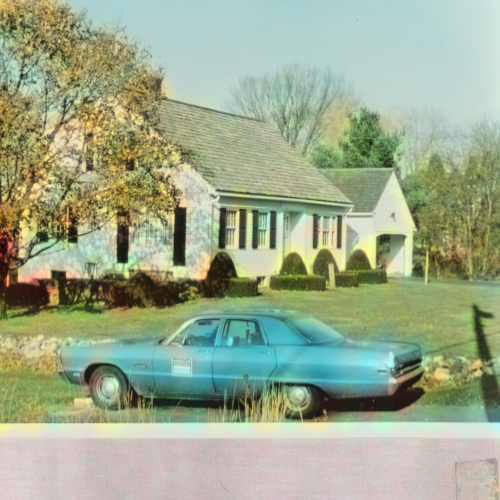}
		\caption{$\varphi_1$, PSNR: $27.19$.}\label{house_hq}
	\end{subfigure}
	\begin{subfigure}[t]{0.24\textwidth}
		\includegraphics[width = \textwidth]{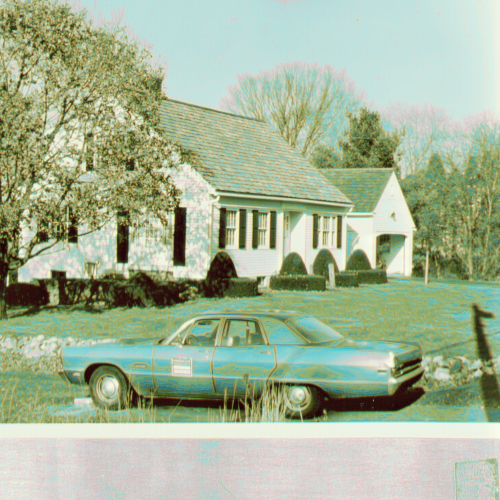}
		\caption{\cite{QKL10}, PSNR: $22.49$.}\label{house_kang}
	\end{subfigure}
	\caption{Image colorization.
	\subref{house_orig} Original image.
	\subref{house_lost} Corrupted image where $99\%$ of the color information (chromaticity) is lost. 
	Colorization using 
	\subref{house_hq} inpainting of the chromaticity with $\varphi_1,\ \lambda = 1,\ \varepsilon = 10^{-1}$. 
	\subref{house_kang}  the  method in \cite{QKL10} with parameters  $r=1,p=1,\sigma_1=2,\sigma_2=\infty,\gamma=0$.}
\end{figure}

Our final experiment shows the smoothing of 3D directions in the synthetic image in Fig.~\ref{s2_field_noisy}. 
We use half-quadratic minimization with $\varphi_2$ to obtain Fig.~\ref{s2_field_res}. 
The original pattern is again visible.
%
\begin{figure}[tbp]
	\centering
	\begin{subfigure}{0.32\textwidth}
		\includegraphics[width = \textwidth]{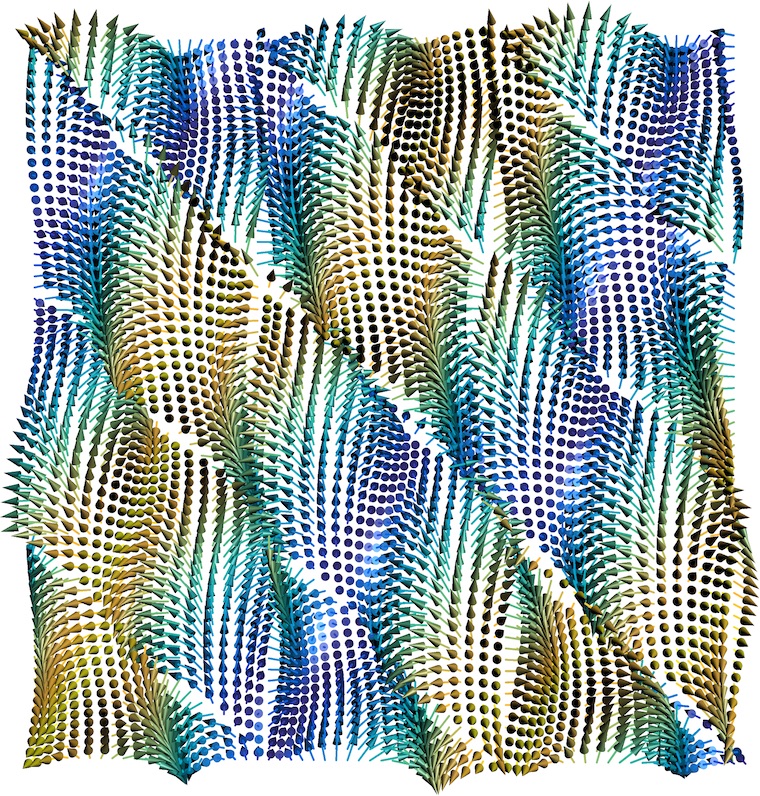}
		\caption{Original $\Ss^2$-field.}\label{s2_field_orig}
	\end{subfigure}
	\begin{subfigure}{0.32\textwidth}
		\includegraphics[width = \textwidth]{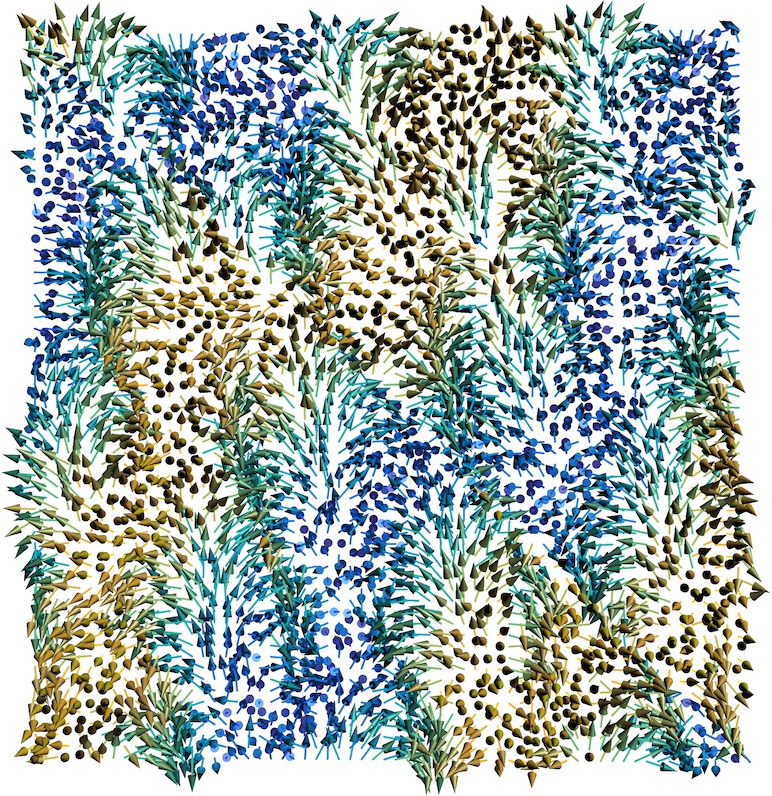}
		\caption{Noisy $\Ss^2$-field.}\label{s2_field_noisy}
	\end{subfigure}
	\begin{subfigure}{0.32\textwidth}
		\includegraphics[width = \textwidth]{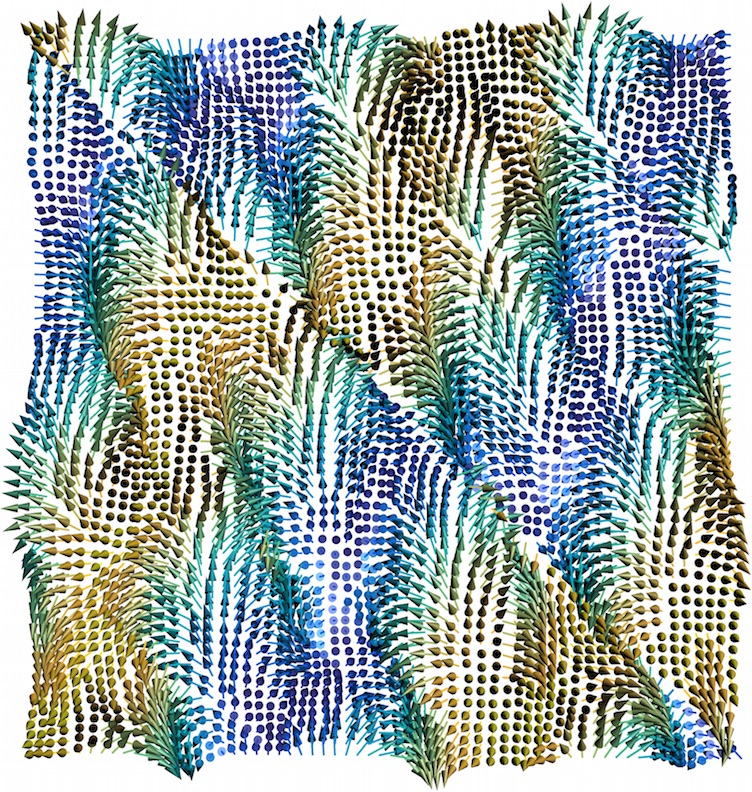}
		\caption{$\varphi_2$.}\label{s2_field_res}
	\end{subfigure}
	\caption{\subref{s2_field_orig} Original $\Ss^2$-field of size \(64\times 64\).
	\subref{s2_field_noisy} Corrupted field by Gaussian noise, $\sigma = 0.1$. 
	\subref{s2_field_res} Restored field with $\varphi_2,\ \lambda=2.6,\ \varepsilon =10^{-1}$, 
	leaving an error $\mathrm{err} = 0.1705$.}
\end{figure}

%
\subsection{${\mathcal P}(3)$-valued data}
\begin{figure}[tbp]
\centering
\begin{subfigure}[t]{0.32\textwidth}
	\centering
	\includegraphics[width=.8\textwidth]{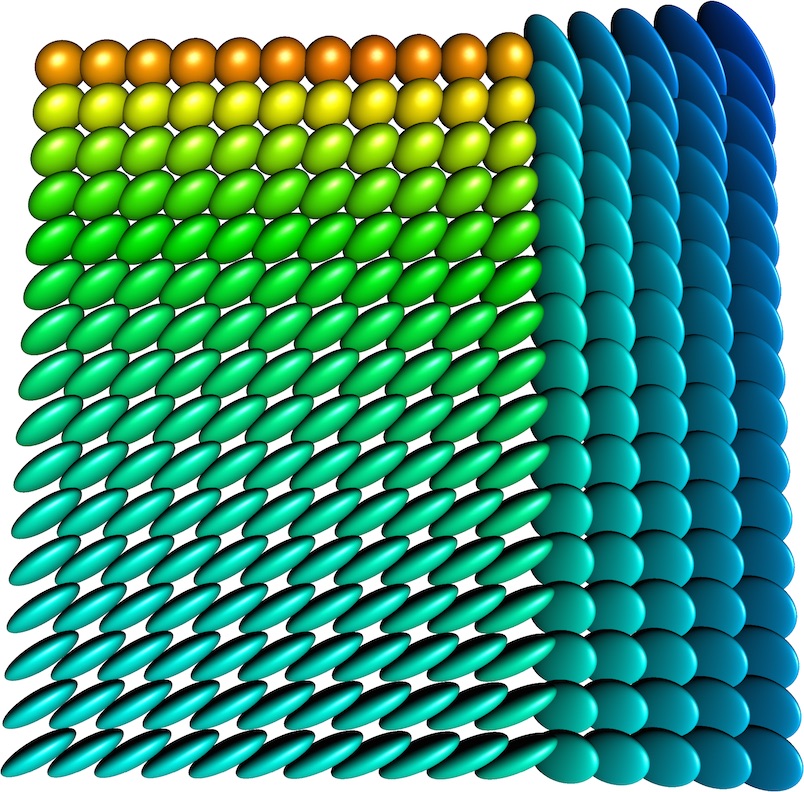}
	\caption{Original image.}\label{subfig:spd:inp:orig}
\end{subfigure}
\begin{subfigure}[t]{0.32\textwidth}
	\centering
	\includegraphics[width=.8\textwidth]{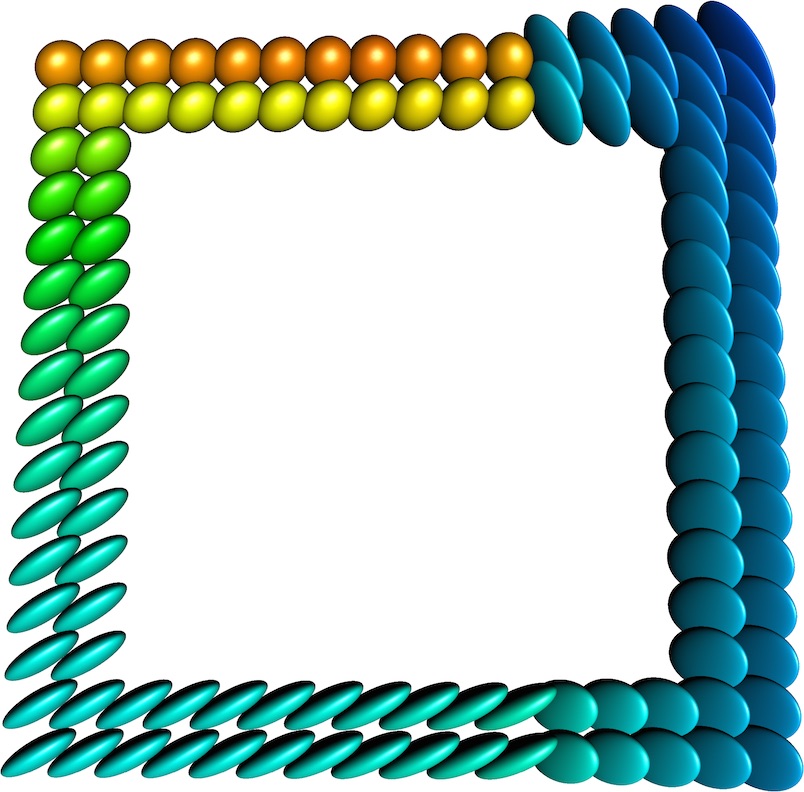}
	\caption{Corrupted image.}\label{subfig:spd:inp:lost}
\end{subfigure}\\
\begin{subfigure}[t]{0.32\textwidth}
	\centering
	\includegraphics[width=.8\textwidth]{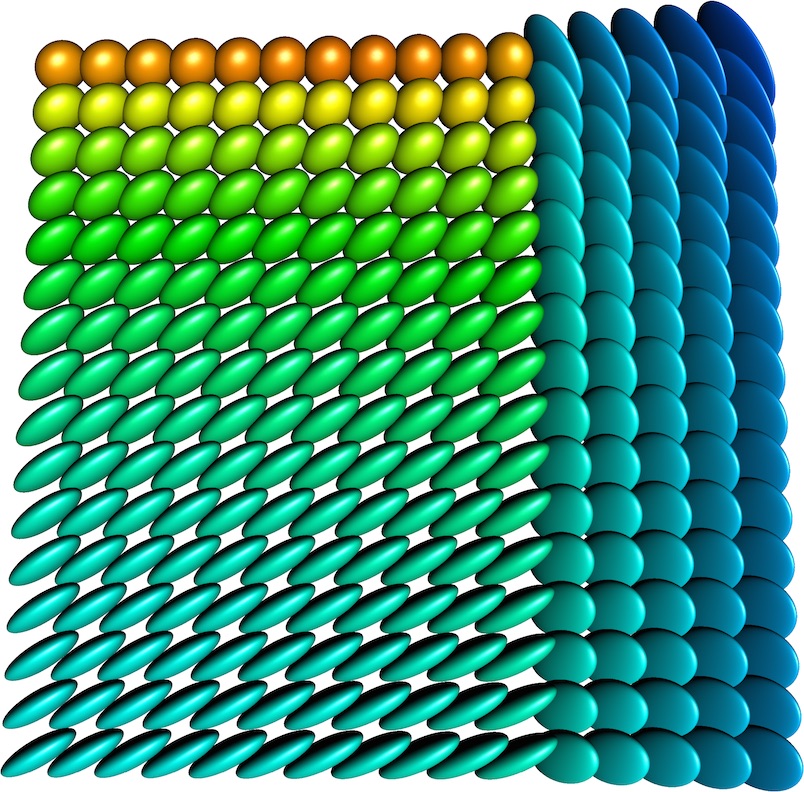}
	\caption{$\varphi_1,\ \varepsilon = 10^{-3}$.}\label{subfig:spd:inp:hqbest}
\end{subfigure}
\begin{subfigure}[t]{0.32\textwidth}
	\centering
	\includegraphics[width=.8\textwidth]{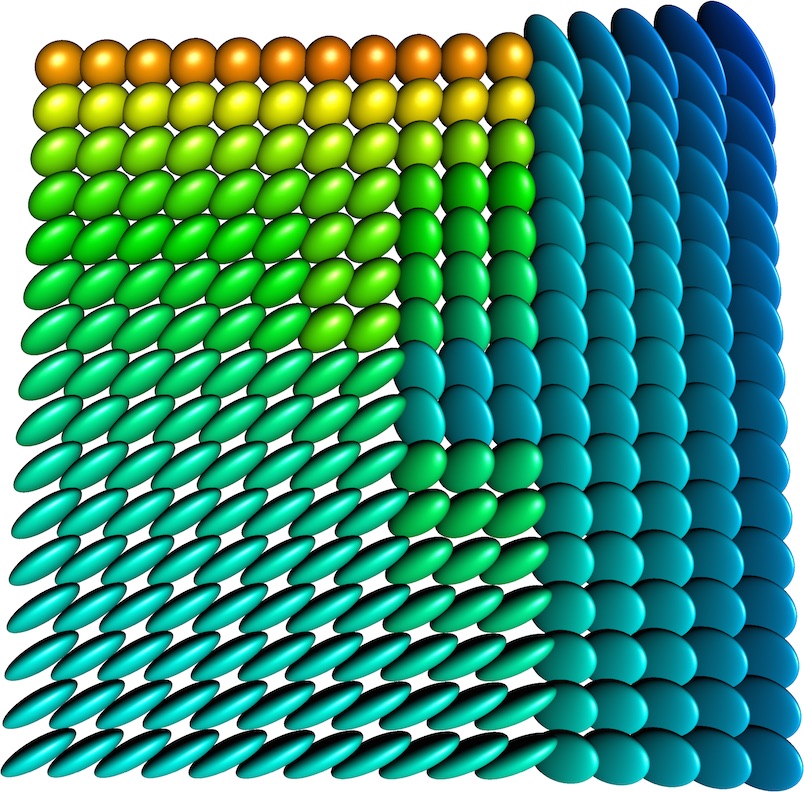}
	\caption{$\varphi_1,\ \varepsilon = 10^{-6}$.}\label{subfig:spd:inp:hq}
\end{subfigure}
\begin{subfigure}[t]{0.32\textwidth}
	\centering
	\includegraphics[width=.8\textwidth]{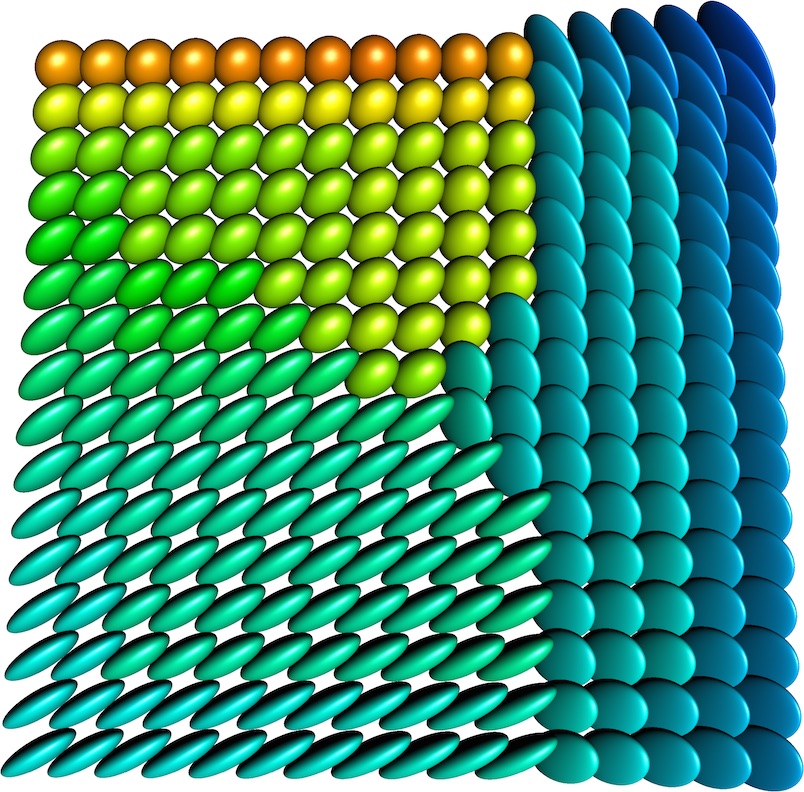}
	\caption{TV, $\alpha=3\times10^{-3}, \lambda = 20$.}\label{subfig:spd:inp:tv}
\end{subfigure}
\caption{Inpainting of an $\mathcal{P}(3)$-valued image.
\subref{subfig:spd:inp:orig} Original image.
\subref{subfig:spd:inp:lost} Image with unknown areas. 
Half-quadratic based inpainting 
\subref{subfig:spd:inp:hqbest} yields a perfect result. 
Decreasing $\varepsilon$  like in \subref{subfig:spd:inp:hq} yields a result closer to 
\subref{subfig:spd:inp:tv} TV.}
\end{figure}
Our first example illustrates the inpainting capabilities of the half-quadratic minimization method by 
an artificial example.
The ${\mathcal P}(3)$-valued image of size $16\times16$ in Fig.~\ref{subfig:spd:inp:orig} has a jump at $\frac23$ in $x$ direction. 
We destroy a center square of size $12\times12$, see Fig.~\ref{subfig:spd:inp:lost}. 
We are able to reconstruct the inpainting area nearly perfectly by using $\varphi_1$ and $\varepsilon=10^{-3}$, 
see Fig.~\ref{subfig:spd:inp:hqbest}. 
Nevertheless decreasing $\varepsilon$ introduces more and more staircasing, cf.\ Fig.~\ref{subfig:spd:inp:hq}. 
This resembles the TV case shown in Fig.~\ref{subfig:spd:inp:tv}, using the model from \cite{WDS14}.

\begin{figure}[tbp]
	\begin{subfigure}{.5\textwidth}\centering
		\includegraphics[width=.9\textwidth]{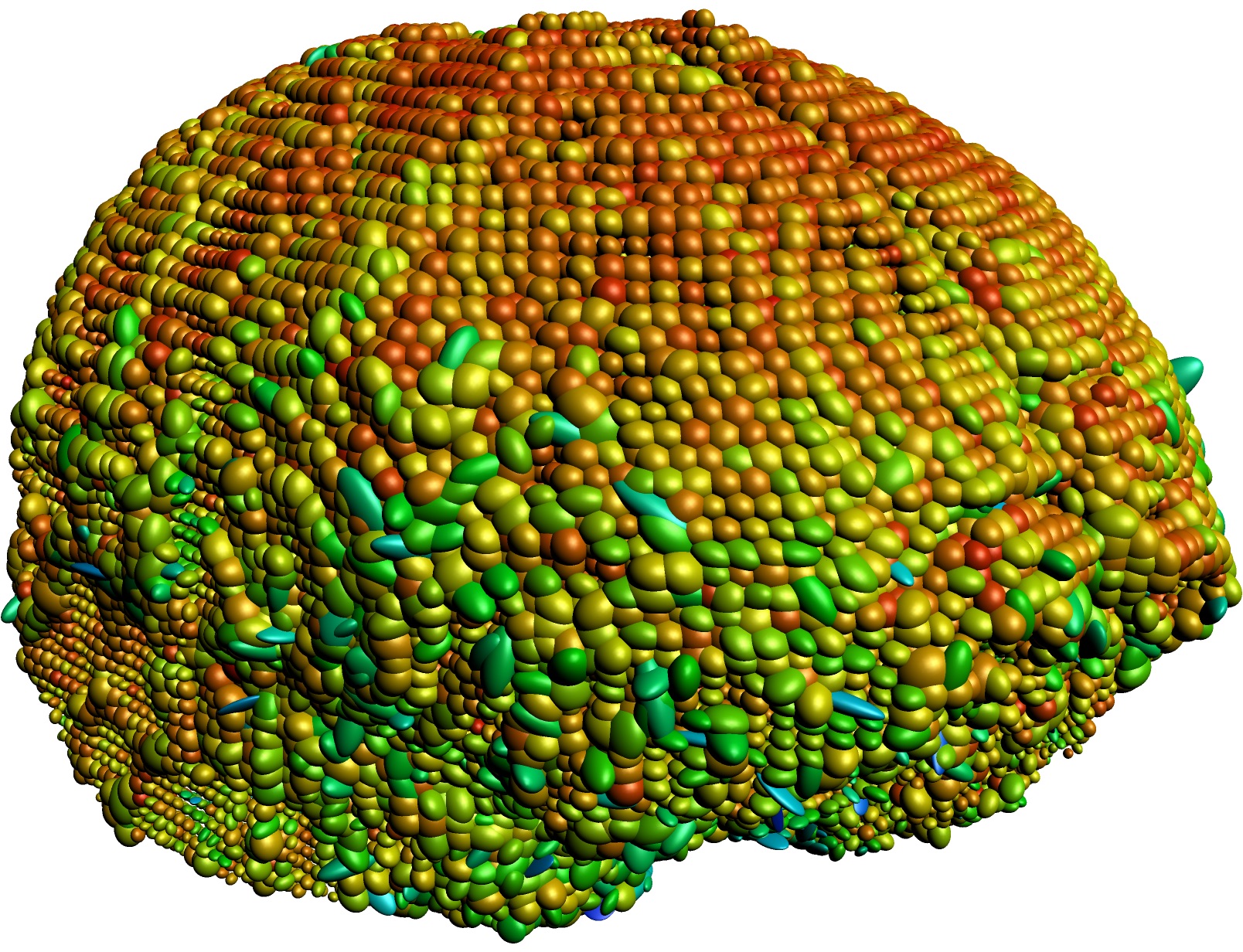}
		\caption{Original ``Camino'' data set.}\label{subfig:CaminoC:orig}
	\end{subfigure}
	\begin{subfigure}{.5\textwidth}\centering
		\includegraphics[width=.9\textwidth]{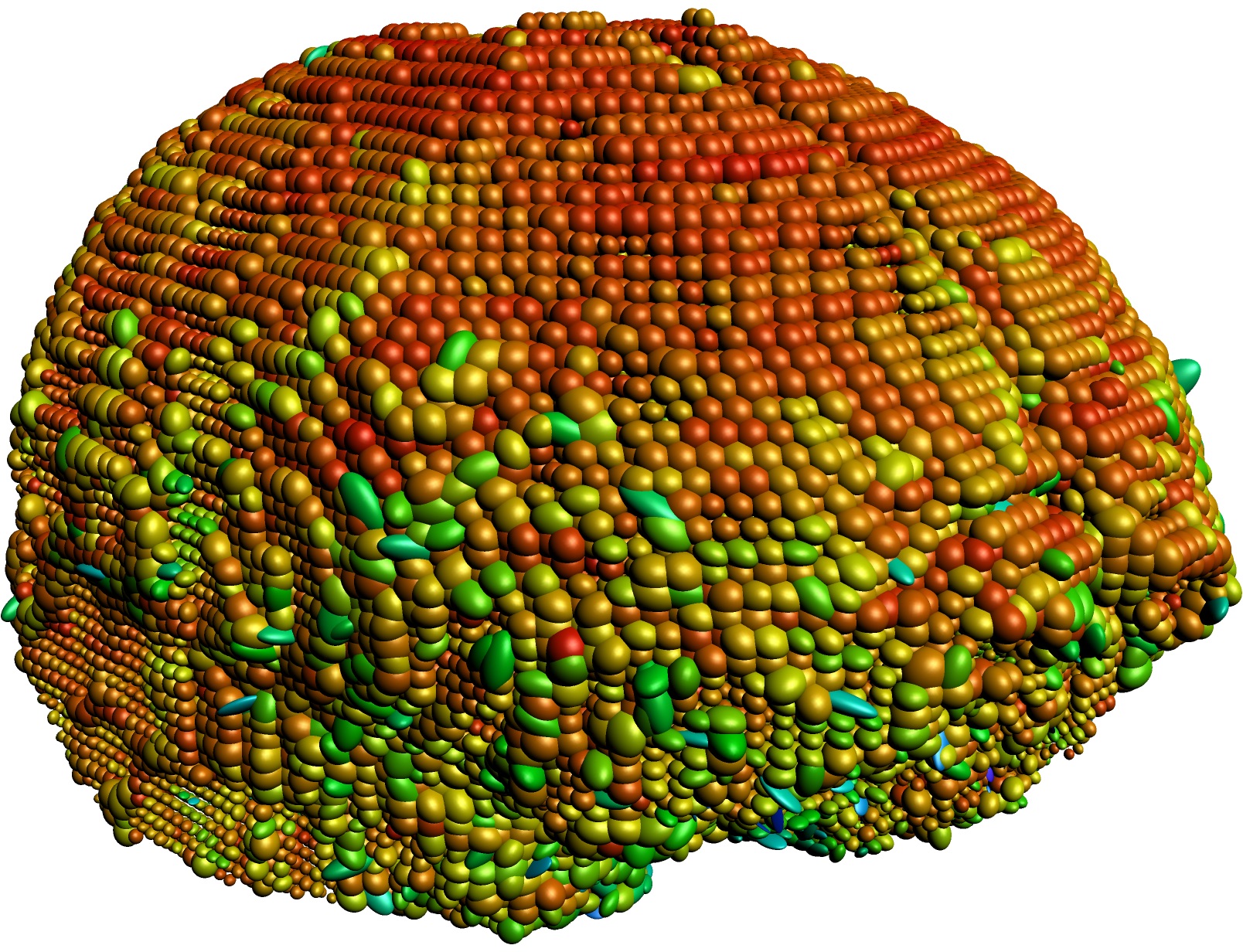}
		\caption{$\varphi_1$, \(\alpha=0.1\), \(\varepsilon=10^{-3}\).}\label{subfig:CaminoC:reg}
	\end{subfigure}
	\caption{Denoising with half-quadratic minimization of every traversal plane of the ``Camino'' data set 
	(DT-MRI of the human head).}\label{fig:Camino}
\end{figure}

An important application of ${\mathcal P}(3)$-valued image denoising is Diffusion Tensor Magnetic Resonance Imaging (DT-MRI).
The Camino project\footnote{see~\href{http://cmic.cs.ucl.ac.uk/camino/}{http://cmic.cs.ucl.ac.uk/camino}}\cite{Camino} 
provides a DT-MRI dataset of the human head
which is freely available.\footnote{follow the tutorial 
at \href{http://cmic.cs.ussscl.ac.uk/camino//index.php?n=Tutorials.DTI}{http://cmic.cs.ucl.ac.uk/camino//index.php?n=Tutorials.DTI}}
The complete data is given as a 3D image
\(\tilde f = \bigl(\tilde f_{i,j,k}\bigr)\in\mathcal P(3)^{112\times112\times50}\),
where we apply the half-quadratic minimization to each of the traversal planes \(k\in\{1,\ldots,50\}\). 
The original dataset, cf.
Fig.~\ref{subfig:CaminoC:orig}, is plotted using the anisotropy index relative
to the Riemannian distance~\cite{MoBa06} normalized onto \([0,1)\) and colored in hue. 
The half-quadratic minimization is used with $\varphi_1$ and the
parameters \(\lambda=0.1\), \(\varepsilon=10^{-3}\) 
and a maximum change between two successive iterations being $10^{-12}$ as a stopping criterion. 
We obtain the result shown in Fig.~\ref{subfig:CaminoC:reg}. For the complete dataset of \(168{,}169\) nonzero matrices, 
the algorithm needed $2{,}492$ seconds to compute the result.
%
\subsection{$\operatorname{SO}(3)$-valued data}
Processing images with $\operatorname{SO}(3)$-valued entries is fundamental in the analysis
of polycrystalline materials by means of Electron Backscattered Diffraction
(EBSD), cf.~\cite{KuWrAdDi93,ASWK93}. Since the microscopic grain structure
affects macroscopic attributes of materials such as ductility, electrical and
lifetime properties, there is a growing interest in the grain structure of
crystalline materials such as metals and minerals. EBSD provides us for each
position on the surface of a specimen with a so called Kikuchi pattern, which
allows the identification of the structure (material index) and the
orientation of the crystal at this position relative to a fixed coordinate
system ($\operatorname{SO}(3)$ value). Since the atomic structure of a crystal is invariant
under its specific symmetry group $S \subset \operatorname{SO}(3)$ the orientation is only
given as an equivalence class $[m_{0}] = \{m_{0} s\mid s \in S\} \in \operatorname{SO}(3)/S$,
$m_{0} \in \operatorname{SO}(3)$.

Fig.~\ref{subfig:so3:magsample} displays a typical EBSD image consisting of lattice
orientations of deformed Magnesium collected by~\cite{ShiLec14}. Each pixel of
the image corresponds to a position on the surface of a Magnesium
specimen. The color of the pixels is chosen corresponding to the orientation
measured at this position according to the following color mapping: for a
fixed vector $\vec{r} \in \mathbb S^2$ we consider the mapping
$\Phi \colon \operatorname{SO}(3)/S \to \mathbb S^{2}/S$, $[m] \mapsto [m^{-1} \vec r]$. Next we
colorize the quotient $\mathbb S^{2}/S$ as it is depicted in
Fig.~\ref{subfig:so3:sphere}. From the colorization scheme the symmetry group
$S \subset \operatorname{SO}(3)$ of Magnesium becomes visible which has six rotations with
respect to a 6-folded axis ($k \pi/3 $, $k=1,\ldots,6$ rotations around $c$
direction) and six rotations with respect to 2-folded axis $a_{1}$, $a_{2}$
perpendicular to that.
\begin{figure}[tbp]
	\centering
	\begin{subfigure}{0.35\textwidth}
		\centering
		\includegraphics[width=0.9\textwidth]{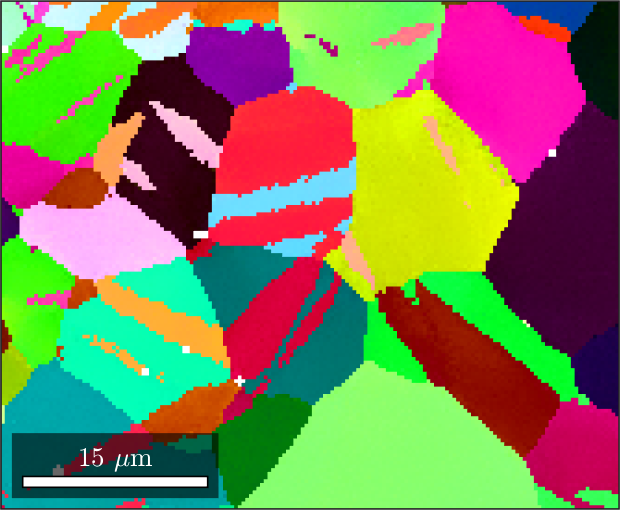}
		\caption{Visualization of a Magnesium sample by \cite{BaHiSc11}.}\label{subfig:so3:magsample}
	\end{subfigure}
	\quad
	\begin{subfigure}{0.25\textwidth}
		\centering
		\includegraphics[width=0.9\textwidth]{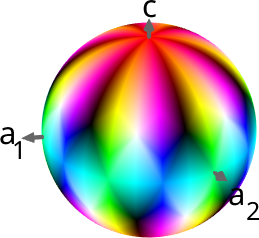}
		\caption{Colorization of the sphere.}\label{subfig:so3:sphere}
	\end{subfigure}\quad	
	\begin{subfigure}{0.20\textwidth}
		\centering
		\includegraphics[width=0.9\textwidth]{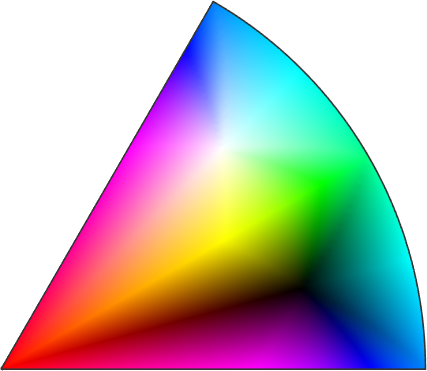}
		\caption{Spherical triangle.}\label{subfig:so3:triangle}
		\includegraphics[width=0.9\textwidth]{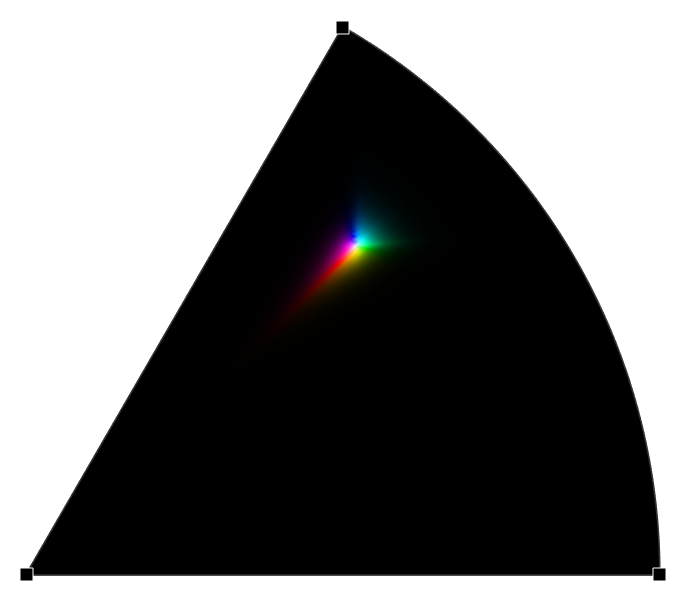}
		\caption{Grain.}\label{subfig:so3:small:triangle}
	\end{subfigure}
	\caption{\subref{subfig:so3:magsample} The raw EBSD data of a Magnesium sample, 
	\subref{subfig:so3:sphere} the colorization of the sphere used to assign to each rotation a certain color 
	according to the mapping $\operatorname{SO}(3)/S \ni m \mapsto m^{-1} (0,0,1)^{T} \in \mathbb S^{2} /S$, \subref{subfig:so3:triangle} 
	colorization of a spherical triangle,  
	\subref{subfig:so3:small:triangle} stretched colorization for one grain.
	}\label{ebsdraw}
\end{figure}

EBSD images usually consist of regions with similar orientations called
grains. For certain macroscopic properties the pattern of orientations within
single grains is of importance~\cite{BaHiJuPaScWe10}, e.g., for the
computation of geometrically necessary
dislocations~\cite{nye53,SuAdKi00} the gradient of the rotations within single
grains has to be determined. As the rotation determination by Kikuchi
patterns is sometimes fragile, the rotation valued images determined by EBSD
are often corrupted by noise and suffer from missing data so that denoising and
inpainting techniques have to be applied~\cite{GuAg10}.
For detecting grains in the raw EBSD data we applied a thresholding
algorithm~\cite{BaHiSc11}. Fig.~\ref{grain1} displays a single
grain with its rotations. Since the rotations vary very little within a single
grain we applied a sharper colorization, cf.~Fig.~\ref{subfig:so3:small:triangle} 
to make the noise and the rotation gradient visible.

We want to apply half-quadratic minimization to denoise EBSD images.
Since crystallographic symmetry groups are finite the quotient $\operatorname{SO}(3)/S$ is
locally isomorphic to $\operatorname{SO}(3)$. In particular, the formulas given in
Appendix~\ref{app:so3} can be applied.  In a first experiment we apply half-quadratic minimization using
$\varphi_1$ to the rotation-valued image depicted in Fig.~\ref{grain1}
which leads to the smooth image in Fig.~\ref{grain2}. 
In a second experiment we 
randomly removed $30\%$ of the data  shown in Fig.~\ref{grain3}. Using
half-quadratic minimization for jointly inpainting and denoising the image we
obtain the result shown in Fig.~\ref{grain4} which looks very similar to
those in Fig.~\ref{grain2}.
In a third experiment we applied half-quadratic minimization simultaneously to
several grains. 
The challenge from the mathematical point of view is that
$\operatorname{SO}(3)$ is not an Hadamard manifold. Convergence can be guaranteed only locally,
which is the case for single grains but may be not true when considering several grains
simultaneously. From the practical point of view this case is
especially interesting as missing data usually occur at grain boundaries,
i.e., between grains.  However, for our data we have got promising results.
Fig.~\ref{grains1} shows the grain from the previous example
(pink color) and two other grains in its neighborhood.  Note that the top
middle area (light green) and top right area (brown-green) belong to the same
grain. Pixels with missing data are plotted white.  Half-quadratic
minimization restoration with $\varphi_1$ improves the image as can be seen in
Fig.~\ref{grains2}. For our last experiment we again randomly remove $30\%$
of the data, cf.~Fig~\ref{grains3}. Restoration with $\varphi_1$ leads to the result in Fig.~\ref{grains4}, which is again hardly to distinguish from Fig.~\ref{grains2}.
Using $\varphi_3$ leads to even better results as depicted in Fig.~\ref{grains5}. We can adjust the smoothing in such a way
that the edge distinguishing two grains is not smoothed, while smaller
rotation changes are smoothed. This is advantageous in the large top grain.
Finally Fig.~\ref{grains6} shows the zoom to the (pink) grain from Fig.
\ref{grain} with adapted color map.

\section{Conclusions} \label{conclusions}
We adapted the principles of half-quadratic minimization to the setting of complete, connected Riemannian manifolds. 
In particular, the notation of the $c$-transform provides an interesting point of view.
For Hadamard manifolds we proved the existence and uniqueness of the minimizer of the  corresponding  functionals as well as the
convergence for the alternating minimization algorithm under moderate assumptions.
The multiplicative half-quadratic minimization method resembles a quasi-Newton method  \cite{NC07} and appears to be very efficient in our numerical examples.
There are numerous applications of the approach.
In this paper, images having values in a manifold such as the 2-sphere or the symmetric positive definite matrices were denoised. 
The method was also used for inpainting missing information into images consisting of either rotation matrices or symmetric positive definite matrices. 
In the chromaticity-brightness color model, the inpainting technique was applied to the task of colorization.
The method has further potential in EBSD.

Topics of future research are the derivation of convergence proofs for more general manifolds
under special assumptions on the local behavior of the data. 
Furthermore, different data terms must be included for other applications,
and the inclusion of higher order differences into the regularization term of the model is of interest. 
\begin{figure}[tbp]
	\centering
	\begin{subfigure}{0.49\textwidth}
		\centering
	\includegraphics[width=.7\textwidth]{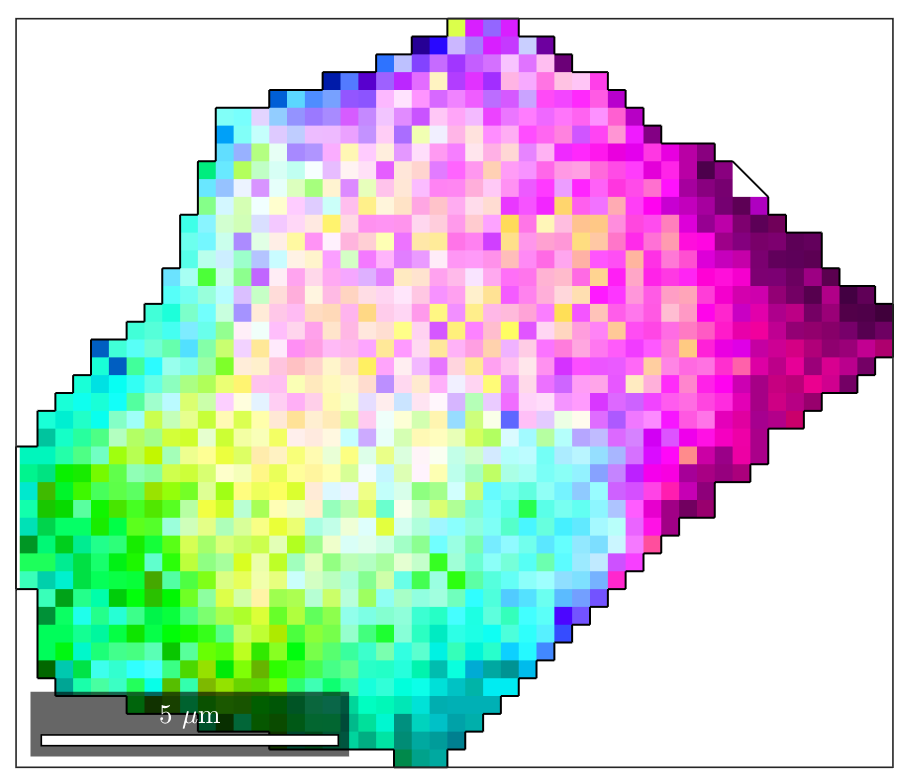}
		\caption{Original grain.}\label{grain1}
	\end{subfigure}	
	\begin{subfigure}{0.49\textwidth}
		\centering
		\includegraphics[width=.7\textwidth]{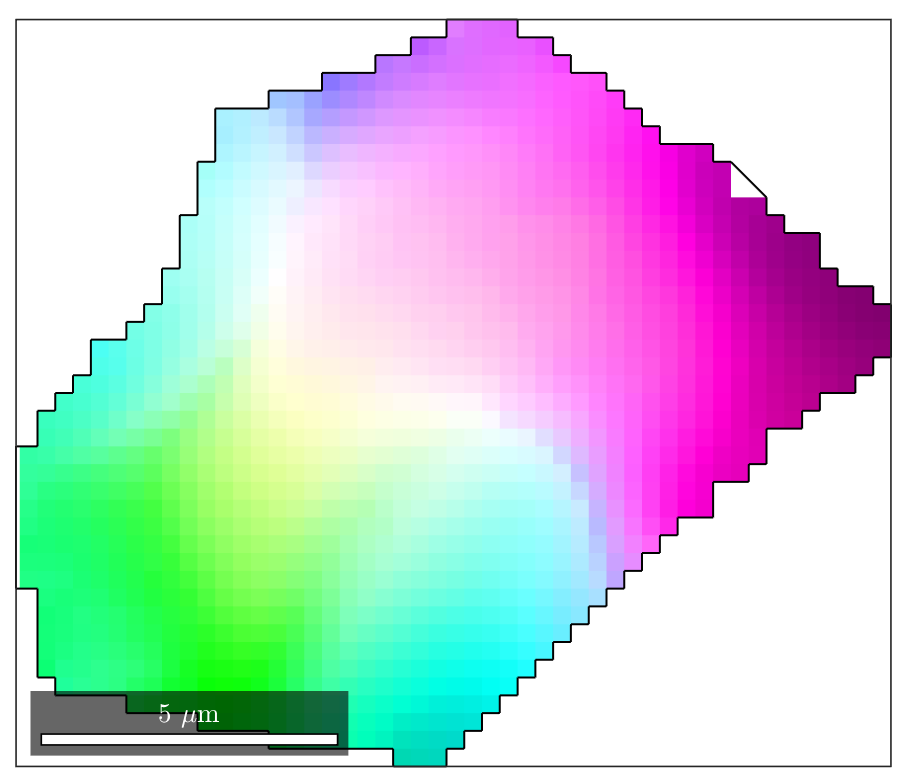}
		\caption{Grain smoothed with $\varphi_1$.}\label{grain2}
	\end{subfigure}
	\begin{subfigure}{0.49\textwidth}
		\centering
		\includegraphics[width=.7\textwidth]{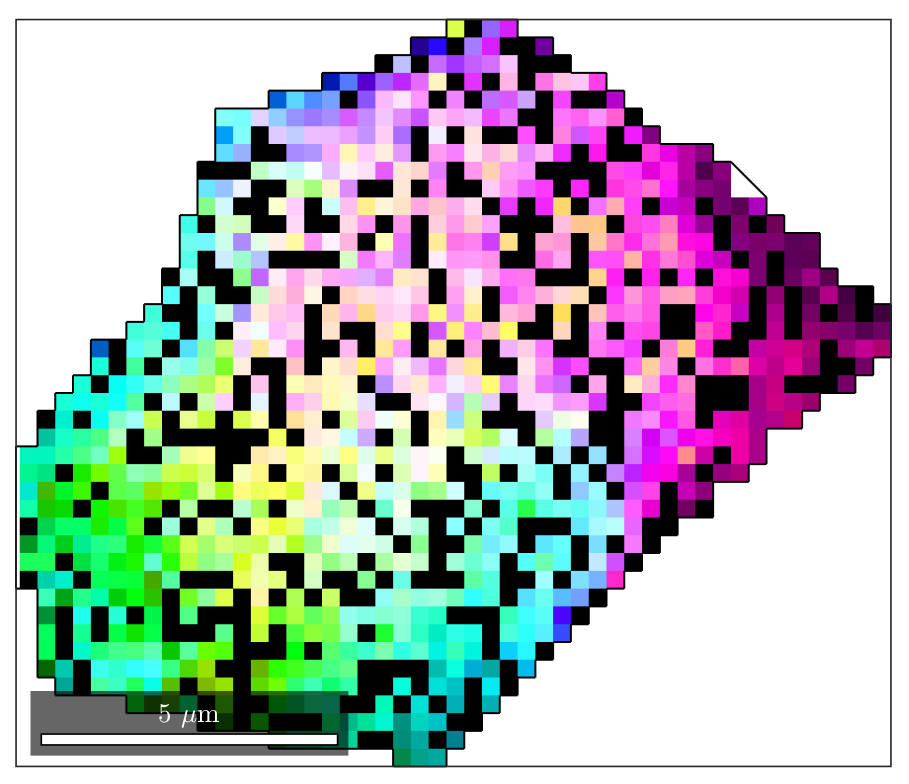}
		\caption{Grain with $30\%$ lost data.}\label{grain3}
	\end{subfigure}
	\begin{subfigure}{0.49\textwidth}
		\centering
		\includegraphics[width=.7\textwidth]{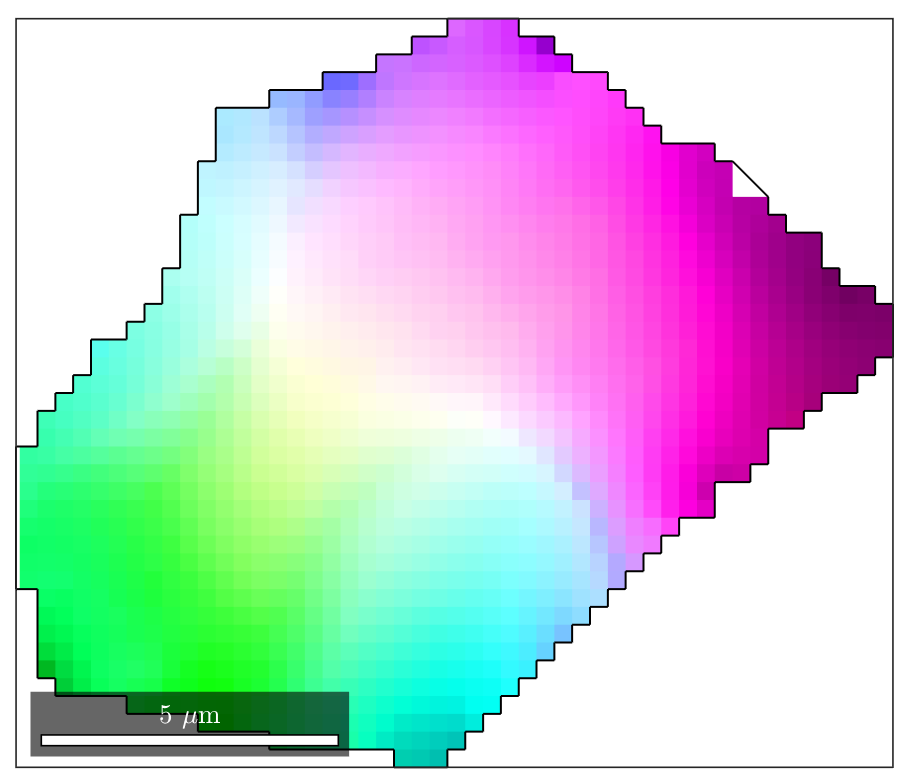}
		\caption{Grain restored with $\varphi_1$.}\label{grain4}
	\end{subfigure}
	\caption{\subref{grain1} Original grain from a Magnesium specimen (grain boundary in black), 
	\subref{grain2} smoothed grain 	using $\varphi_1$ with parameters $\lambda= 0.1$, $\varepsilon = 10^{-2}$,
		\subref{grain3} grain with $30\%$ lost data, marked in white, \subref{grain4} inpainted and smoothed
		grain with $\varphi_1$ and the same parameters.}\label{grain}
\end{figure}
\begin{figure}[tbp]
	\centering
	\begin{subfigure}[t]{0.32\textwidth}
		\includegraphics[width=\textwidth]{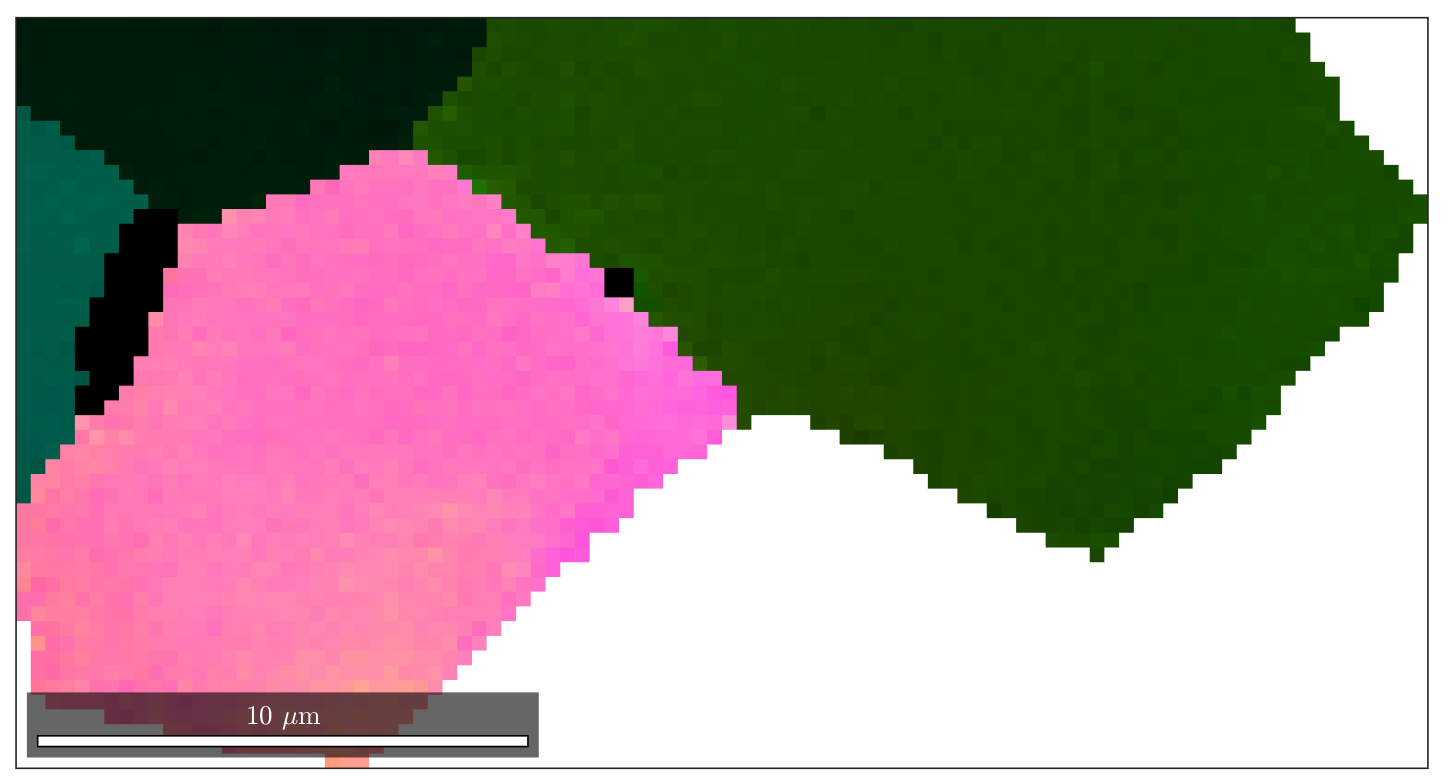}
		\caption{Three grains.}\label{grains1}
	\end{subfigure}
	\begin{subfigure}[t]{0.32\textwidth}
		\includegraphics[width=\textwidth]{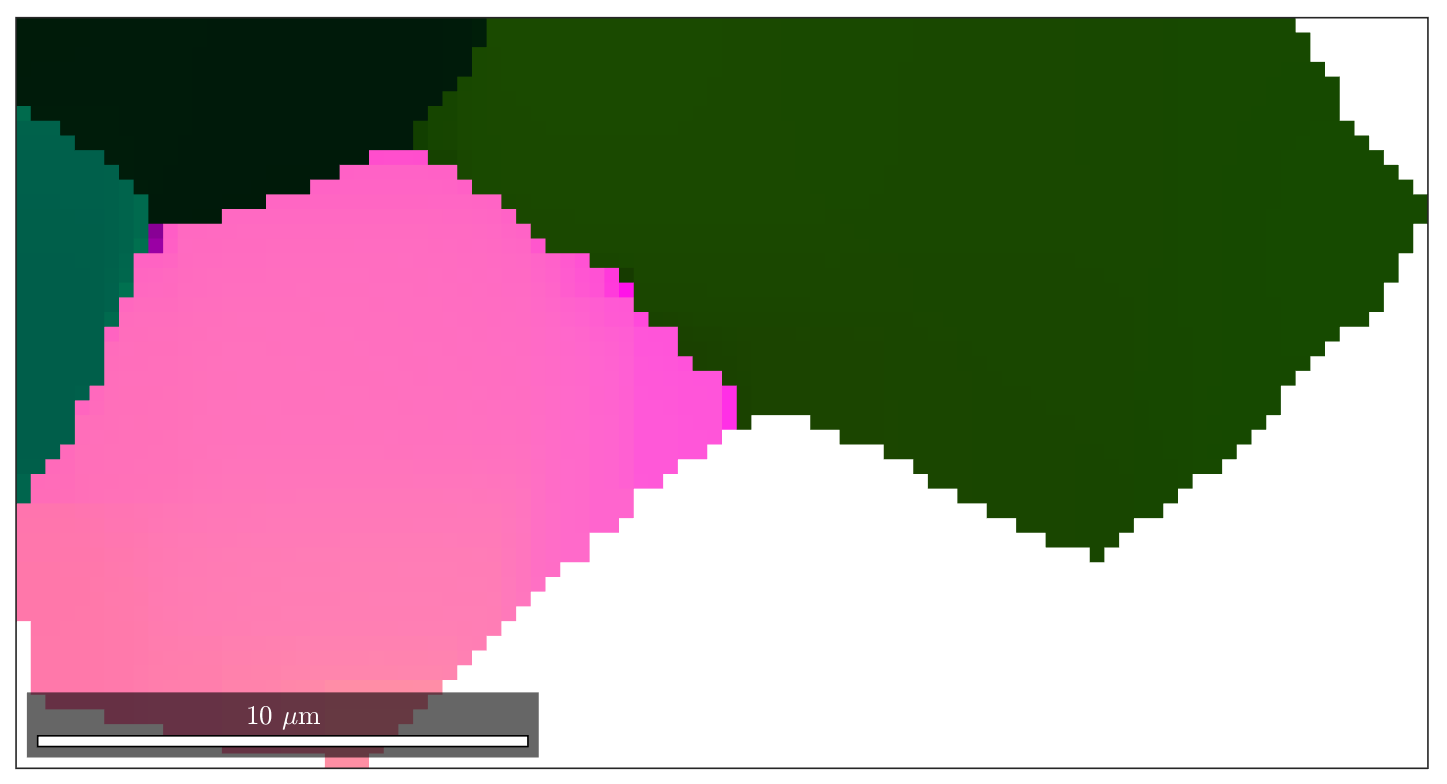}
		\caption{Grains smoothed with $\varphi_1$.}\label{grains2}
	\end{subfigure}
	\begin{subfigure}[t]{0.32\textwidth}
		\includegraphics[width=\textwidth]{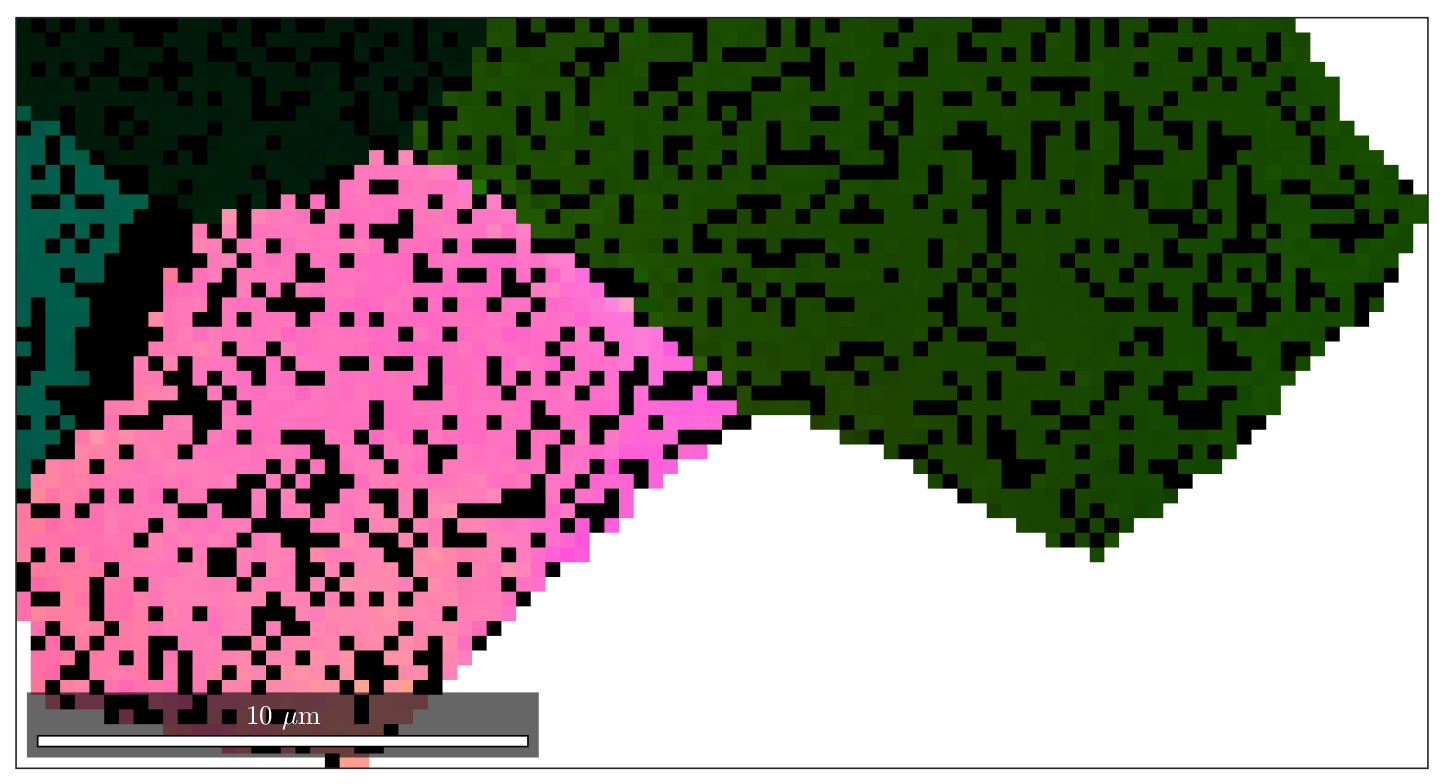}
		\caption{Grains with $30\%$ lost data.}\label{grains3}
	\end{subfigure}
	\begin{subfigure}[t]{0.32\textwidth}
		\includegraphics[width=\textwidth]{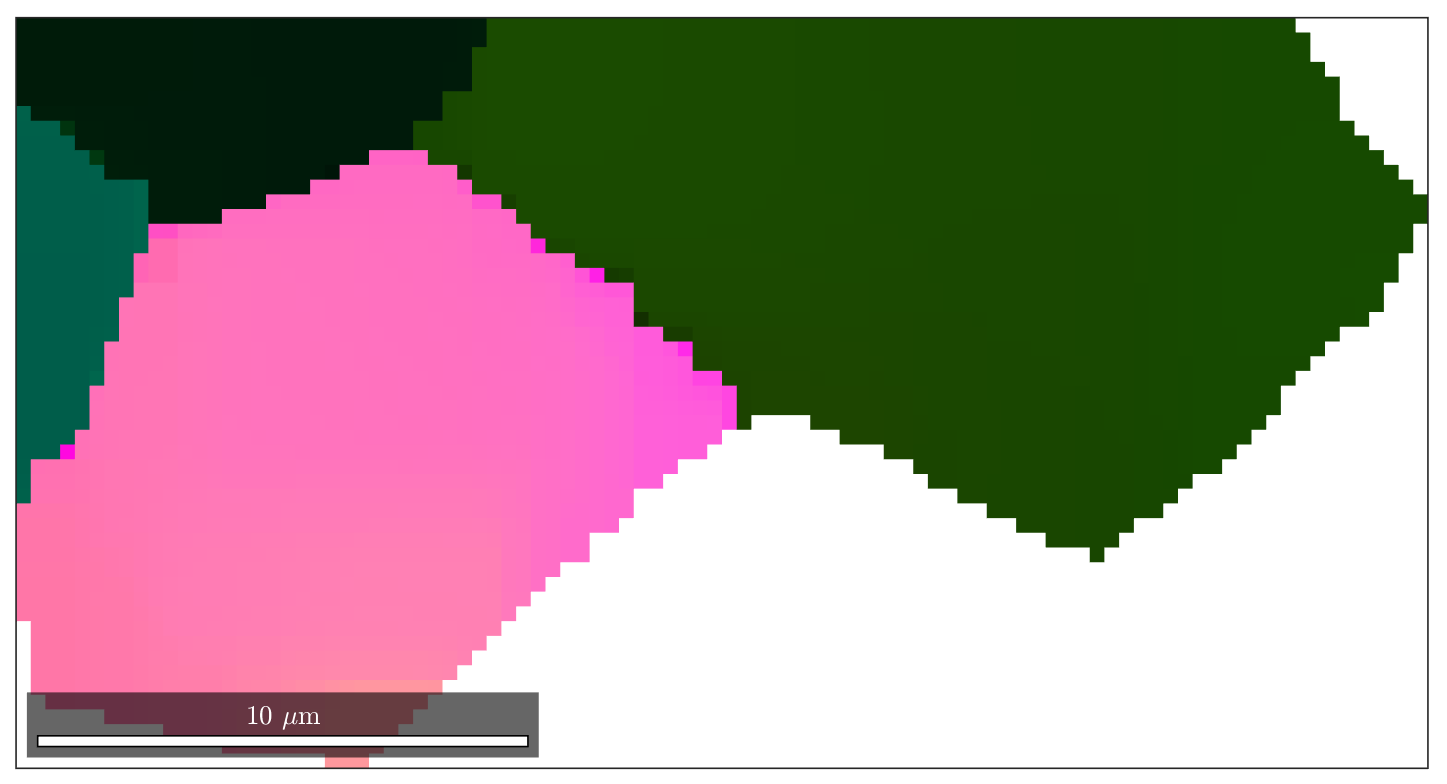}
		\caption{Grains restored with $\varphi_1$.}\label{grains4}
	\end{subfigure}
	\begin{subfigure}[t]{0.32\textwidth}
		\includegraphics[width=\textwidth]{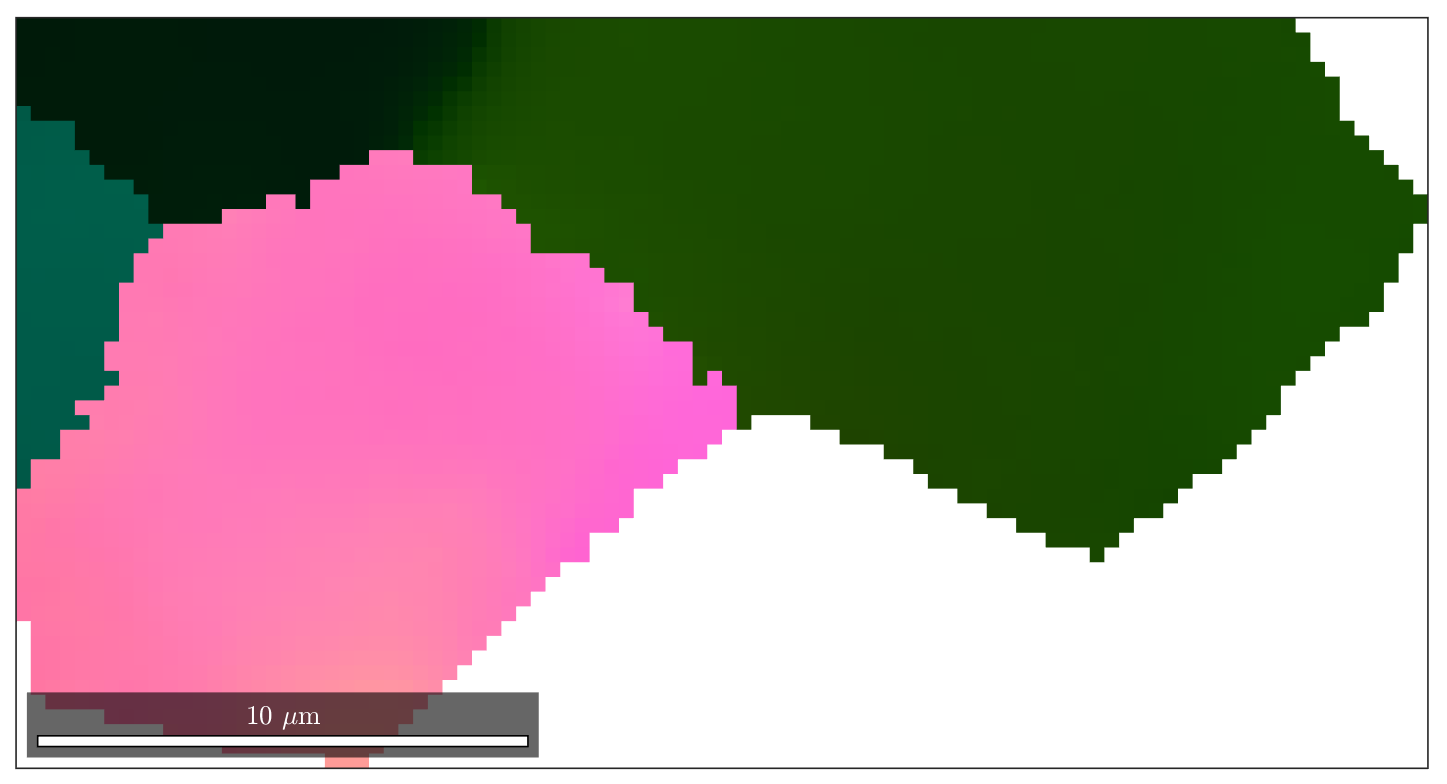}
		\caption{Grains restored with $\varphi_3$.}\label{grains5}
	\end{subfigure}
	\begin{subfigure}[t]{0.32\textwidth}
		\includegraphics[width=\textwidth]{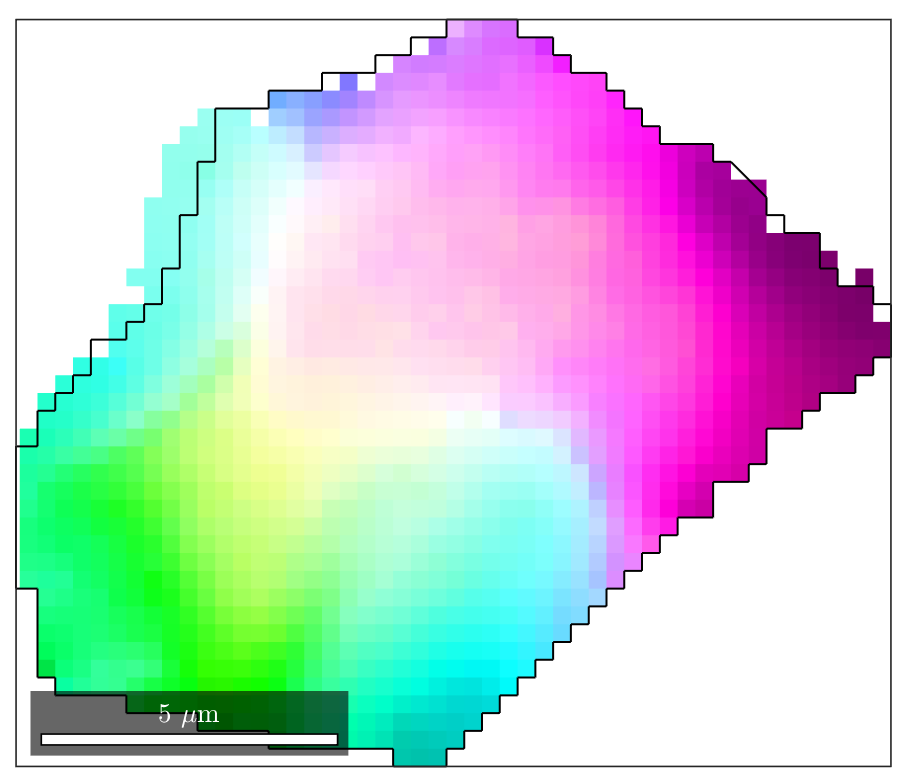}
		\caption{Grain from Fig.~\ref{grain}.}\label{grains6}
	\end{subfigure}
	\caption{\subref{grains1} Three grains from a Magnesium specimen with missing data in white.
	\subref{grains2} Smoothed (and ipainted) grains with $\varphi_1$ and parameters $\lambda= 0.15$ $\varepsilon = 10^{-4}$.
	\subref{grains3} Grains with $30\%$ data lost.
	Restored grains from (c) with
	\subref{grains4} $\varphi_1$ and parameters  $\lambda= 0.05$ $\varepsilon = 0.5\times10^{-4}$, 
	\subref{grains5} $\varphi_3$ and parameters $\lambda=0.1,\ \varepsilon=\sqrt{30}$.
	\subref{grains6} Grain from Fig.~\ref{grain} after restoration with $\varphi_3$ (original boundary in black).}\label{grains}
\end{figure}
%
\appendix
\section{Proofs} \label{app:proofs}

\begin{proof}[Proof of Proposition \ref{prop_1}.]
1. In the additive case we have
\begin{align*}
\varphi^c(s) 
&\coloneqq \inf_{t \in \mathbb R} 
\biggl\{ 
\frac12 \Bigl(\sqrt{a} t - \frac{1}{\sqrt{a}} s \Bigr)^2 - \varphi(t)
\biggr\}
= \inf_{t \in \mathbb R} 
\left\{ 
-st + \tfrac12 a t^2 -  \varphi(t)
\right\} + \frac{1}{2a} s^2\\
&= - \sup_{t \in \mathbb R} 
\Bigl\{ 
 st - \underbrace{\bigl(\tfrac12 a t^2 - \varphi(t) \bigr)}_{\Phi}
\Bigr\} + \frac{1}{2a} s^2\\
&= - \Phi^*(s) + \frac{1}{2a} s^2
\end{align*}
By assumption on $\Phi$ we know that $\Phi = \Phi^{**}$ which implies
\begin{align*}
\varphi^{cc}(t) 
&= \inf_{s \in \mathbb R} 
\bigl\{ 
-st + \tfrac{1}{2a} s^2 - \varphi^c(s) 
\bigr\} + \tfrac12 a t^2
=
\inf_{s \in \mathbb R} 
\bigl\{ -ts + \Phi^*(s) \bigr\} + \tfrac12 a t^2
\\
&= - \Phi(t) + \tfrac12 a t^2 = \varphi(t).
\end{align*}
This finishes the proof of i). The function 
\[h(t) \coloneqq c(t,s) - \varphi(t) = \frac12 a t^2 - \varphi(t) -  s t + \frac{1}{2a} s^2\]
is continuous, convex and by \eqref{existence_add} coercive so that the global minimizer in \eqref{dual_1}
is attained for $0 = h'(t) = at - \varphi'(t) -  s$, i.e., for $\bigl(t,at - \varphi'(t)\bigr)$ which proves ii).
\\
2. In the multiplicative case we obtain, since $\varphi$ is even,
\begin{align*}
\varphi^c(s) 
&\coloneqq \inf_{t \in \mathbb R} 
\bigl\{ 
t^2 s - \varphi(t)
\bigr\}
= \inf_{t \ge 0} 
\bigl\{ 
t^2 s - \varphi(t)
\bigr\}\\
&=
\inf_{r \ge 0} 
\bigl\{ 
r s - \varphi(\sqrt{r})
\bigr\}
=
-\sup_{r \ge 0} 
\Bigl\{ 
- r s - \bigl(-\varphi(\sqrt{r})\bigr)
\Bigr\}\\
&=- \Phi^*(-s).
\end{align*}
We have
\begin{align*}
\varphi^{cc}(t) &= \varphi^{cc} (-t)
= \inf_{s \in \mathbb R} 
\bigl\{ 
t^2 s - \varphi^c(s) 
\bigr\} .
\end{align*}
so that we can restrict our attention to $t \ge 0$.
By assumption, $\Phi$ is convex and lsc. Thus, 
$\Phi = \Phi^{**}$ and we obtain for $t \ge 0$ that
\begin{align*}
\varphi^{cc}(\sqrt{t}) &= \inf_{s \in \mathbb R}
\bigl\{t s - \varphi^c(s) \bigr\}
= - \sup_{s \in \mathbb R} \bigl\{- t s - \Phi^*(-s) \bigr\}
= - \Phi(t) = \varphi(\sqrt{t}).
\end{align*}
This yields i). To see ii) we first note that
condition \eqref{existence_mult} implies for $s \ge 0$ that the objective function 
in \eqref{dual_1} is coercive such that the infimum is attained.
For $s<0$, we have $\varphi^c(s) = -\infty$. For $s \ge 0$, we obtain
\[
\argmin_{t \ge 0} \bigl\{t^2 s - \varphi(t) \bigr\} 
= 
\Bigl( \argmin_{r \ge 0} \bigl\{r s - \varphi(\sqrt{r} ) \bigr\} \Bigr)^\frac12
\]
By assumption on $\Phi$, the function
 $h(r) \coloneqq rs - \varphi(\sqrt{r})$ is convex in $\mathbb R_{\ge 0}$.
A global minimizer of $h$ is attained
either 
for the solution of 
$0 = h'(r) = s - \frac{1}{2\sqrt{r}} \varphi'(\sqrt{r})$
if this solution is positive 
or
for $r=0$ if $\lim_{r \rightarrow 0+} h'(r) \ge 0$, i.e., $s \ge \frac12\varphi''(0+)$.
Therefore $(t,s) = \bigl(0, \frac12\varphi''(0+) \bigr)$
is a solution.

Finally, the concavity of $\varphi(\sqrt{t})$ for $t \ge 0$ implies
that $\varphi'(\sqrt{t})/(2 \sqrt{t})$ and thus $\varphi(t)/(2 t)$ is decreasing.
Under the additional assumption in iii) we get $s \in \bigl( 0,\frac{\varphi''(0+)}{2} \bigr]$.
\end{proof}
\begin{proof}[Proof of Proposition \ref{prop_2}.] i) 
If ${\mathcal V} = {\mathcal G}$, let
$\lVert d(u,f)\rVert_2 \rightarrow \infty$, 
where 
$d(u,f) \coloneqq \big( d(u_i,f_i) \big)_{i \in {\mathcal G}}$. 
Then 
$F(u) \coloneqq \frac12 \sum_{i \in {\mathcal V}} d^2(u_i,f_i)$ 
goes to infinity 
and the functionals $J_\nu$, $\nu = 1,2$, are coercive.
In the case ${\mathcal V} \not = {\mathcal G}$ choose $i_0 \in  {\cal V}$ and
let $u_0$ be the constant image with entries $f_{i_0}$.
Let $\lVert d(u,u_0) \rVert_2 \rightarrow \infty$.
Assume that $J_\nu (u)$ remains finite, so that
in particular 
$d(f_{i_0},u_{i_0})$ 
and 
$d(u_i,u_j)$, $j \in {\cal N}(i)^+$, $i \in {\cal G}$ are finite.
By the construction of the neighborhoods ${\cal N}(i)^+$ there exists for every $j \in {\cal G}$
a path $i_0,i_1,\ldots,i_{k_j} = j$ with $i_{l+1} \in {\cal N}(i_l)^+$
and
\[
d(f_{i_0},u_j) \le d(f_{i_0},u_{i_0}) + d(u_{i_0},u_{i_1}) + \ldots + d(u_{i_{{k_j}-1}},u_{j}).
\]
Since the right-hand side remains finite this contradicts $\lVert d(u,u_0) \rVert_2 \rightarrow \infty$.
Hence  $J_\nu$ $\nu = 1,2$ are coercive.
\\
ii)
By (D2) we have that $F(u)$ is convex and
strictly convex if ${\mathcal V} = {\mathcal G}$.
If a function $h\colon {\mathcal H}^\kappa \rightarrow \mathbb R$ is convex, then, 
for any geodesic $\gamma\colon [0,1] \rightarrow {\mathcal H}^\kappa$ joining $x,y \in {\mathcal H}^\kappa$, we obtain since $\varphi$ 
is increasing and convex that
\begin{equation} \label{help_1}
\varphi\circ h \bigl(\gamma(t)\bigr) \le \varphi\bigl(t h(x)  + (1-t) h(y)\bigr) \le t (\varphi\circ h) (x) + (1-t) (\varphi\circ h) (y)
\end{equation}
so that $\varphi\circ h$ is convex. If $\varphi$ is strictly convex the last inequality is strong so that
$\varphi\circ h$ is strictly convex.
With $h \coloneqq h_{ij} = d(u_i,u_j): {\mathcal H}^2 \rightarrow \mathbb R$ this implies by (D1) that $J_1$ is convex, resp. strictly convex.
Concerning $J_2$ notice that the convexity of $h_i(x_0,x_i)$, $i=1,\ldots,\kappa-1$, on ${\mathcal H}^2$ 
implies convexity of
$h(x_0, \ldots,h_{\kappa-1}) \coloneqq \bigl( \sum_{i=1}^{\kappa-1} h_i^2(x_0,x_i) \bigr)^\frac12$
on ${\mathcal H}^\kappa$ by
\begin{align*}
h^2\bigl(\gamma(t)\bigr) 
&= \sum_{i=1}^{\kappa-1} h_i^2\bigl( \gamma_0(t),\gamma_i(t) \bigr) 
\le
\sum_{i=1}^{\kappa-1} \Bigl( t h_i\bigl(\gamma_0(0),\gamma_i(0)\bigr) + (1-t) h_i\bigl(\gamma_0(1),\gamma_i(1)\bigr) \Bigr)^2\\
&= 
t^2  h^2\bigl(\gamma_0(0) \bigr) + (1-t)^2 h^2\bigl(\gamma_0(1)\bigr)
+ 2 t(t-1) \sum_{i=1}^{\kappa-1} h_i\bigl( \gamma_0(0),\gamma_i(0) \bigr)  h_i\bigl( \gamma_0(1),\gamma_i(1) \bigr) 
\end{align*}
and by the Schwarz inequality
\begin{align*}
h^2\bigl(\gamma(t)\bigr) &\le
t^2  h^2\bigl(\gamma_0(0) \bigr) + (1-t)^2 h^2\bigl(\gamma_0(1)\bigr) + 2 t(t-1) h\bigl(\gamma_0(0) \bigr) h\bigl(\gamma_0(1) \bigl)\\
&= \Bigl( t  h\bigl(\gamma_0(0) \bigr) + (1-t) h\bigl(\gamma_0(1)\bigr) \Bigr)^2.
\end{align*}

In the case ${\mathcal V} = {\mathcal G}$, strict convexity follows by the strict convexity of the data term 
$\sum_{i \in {\mathcal G} }d^2(f_i,u_i)$ and for strictly convex $\varphi$ by the strict convexity in \eqref{help_1}.
\end{proof}
\begin{proof}[Proof of Theorem \ref{convergence}.]
By Remark \ref{rem_u} we know that
$\lim_{k \rightarrow \infty} {\mathcal J}_\nu \bigl(u^{(k)},v^{(k)}\bigr) =: \bar b$
and that there exists a subsequence $\bigl\{\bigl(u^{(k_j)},v^{(k_j)}\bigr)\bigr\}_j$
which converges to some $(\bar u, \bar v)$. Since ${\mathcal J}_\nu$ is continuous we have 
$\lim_{j \rightarrow \infty} {\mathcal J}_\nu \bigl(u^{(k_j)},v^{(k_j)}\bigr) = {\mathcal J}_\nu (\bar u, \bar v) = \bar b$.
Let
\[
\tilde v \coloneqq s\left( {\rm d}_{\bar u} \right) = \argmin_v {\mathcal J}_\nu(\bar u,v) 
\quad {\rm and} \quad
\tilde u \coloneqq \argmin_u {\mathcal J}_\nu(u,\tilde v). 
\]
The continuity of $s$ and ${\rm d}$ implies that
$\lim_{j \rightarrow \infty} v^{(k_j+1)} = \lim_{j \rightarrow \infty} s\bigl( {\rm d}_{u^{(k_j)} }\bigr) = s({\rm d}_{\bar u}) = \tilde v$
and the continuity of $T(u) \coloneqq \argmin_z {\mathcal J}_\nu \bigl(z,s({\rm d}_u)\bigr)$ that
$\lim_{j \rightarrow \infty} u^{(k_j+1)} = \lim_{j \rightarrow \infty} T\bigl(u^{(k_j)}\bigr) = T(\bar u) = \tilde u$.
By \eqref{abstieg} we conclude
\[
\bar b = \lim_{j \rightarrow \infty} {\mathcal J}_\nu \bigl(u^{(k_j+1)},v^{(k_j+1)}\bigr) = 
{\mathcal J}_\nu(\tilde u, \tilde v) \le {\mathcal J}_\nu(\bar u, \tilde v) \le {\mathcal J}_\nu (\bar u, \bar v) = \bar b.
\]
Thus, ${\mathcal J}_\nu(\tilde u, \tilde v) = {\mathcal J}_\nu(\bar u, \tilde v) = {\mathcal J}_\nu (\bar u, \bar v)$
and since ${\mathcal J}_\nu (\bar u, \cdot)$ and ${\mathcal J}_\nu (\cdot, \tilde v)$ have unique minimizers we obtain that
$\tilde u = \bar u$ and $\tilde v = \bar v$.
Consequently, $\bar v = s({\rm d}_{\bar u})$ and
$\bar u = \tilde u = \argmin_u {\mathcal J}_\nu\bigl(u,s({\rm d}_{\bar u})\bigr)$,
which by Remark \ref{rem:critical} implies 
that $\bar u$ is a minimizer of $J_\nu$.

Assume that $\lim_{k \rightarrow \infty} u^{(k)} \not \rightarrow \bar u$.
Then there exists $\varepsilon > 0$ such that infinitely many $u^{(k_i)}$ not contained in 
the open ball $B_\varepsilon(\bar u)$ of radius $\varepsilon$ centered at $\bar u$.
Since $\bigl\{u^{(k_i)}\bigr\}$  is bounded, there exists a convergent subsequence $u^{(k_{i_j})}$
which converges to some point $u^* \not = \bar u$ in the closed set ${\rm M} \backslash B_{\varepsilon}$.
Then $\lim_{j \rightarrow \infty} J_\nu \bigl( u^{(k_{i_j})} \bigr) = J_\nu(u^*) = J(\bar u)$
which contradicts the fact that $J_\nu$ has a unique minimizer.
\end{proof}

\section{Exponential and Logarithmic Maps}\label{app:B}
\subsection{The Sphere \texorpdfstring{${\mathbb S}^2$}{S2}} \label{A-sphere}
We use the parametrization
\[
x(\theta , \varphi) \coloneqq 
\begin{pmatrix}
\cos \varphi \cos \theta\\
\sin \varphi \cos \theta\\
             \sin \theta
\end{pmatrix}, \quad \theta \in \Bigl(-\frac{\pi}{2}, \frac{\pi}{2}\Bigr), \; \varphi \in [0,2\pi).
\]
Then we have the tangent spaces
\[T_x\bigl(\mathbb S^2\bigr) = T_{x(\theta,\varphi)}\bigl(\mathbb S^2\bigr)\coloneqq \bigl\{\eta \in \mathbb R^{d+1} : \eta^\tT x = 0\bigr\}
= \operatorname{span}\bigl\{e_1(\theta,\varphi), e_2(\theta,\varphi)\bigr\}\]
with the normed orthogonal vectors
\[
e_1(\theta,\varphi) \coloneqq \frac{\partial x}{\partial \theta} = 
\begin{pmatrix}
- \cos \varphi \sin \theta\\
- \sin \varphi \sin \theta\\
             \cos \theta
\end{pmatrix},
\quad
e_2(\theta,\varphi) \coloneqq \frac{1}{\cos \theta} \frac{\partial x}{\partial \varphi} = 
\begin{pmatrix}
- \sin \varphi \\
  \cos \varphi \\
             0
\end{pmatrix}.
\]
The Riemannian metric is just the Euclidean distance in $\mathbb R^3$.
The geodesic distance is given by
$d_{\mathbb S^2}(x_1,x_2) \coloneqq \arccos \langle x_1, x_2\rangle$,
and the exponential map and (locally) its inverse, resp.,  by
\begin{align}
\exp_{x}(t\eta) \coloneqq \cos\bigl(t\lVert\eta\rVert_2\bigr) x + \sin\bigl(t \lVert\eta \rVert_2\bigr) \frac{\eta}{\lVert\eta\rVert_2},\\
\log_{x_1} x_2 \coloneqq \frac{x_2 -  \langle x_1,x_2\rangle x_1}{\lVert x_2 - \langle x_1,x_2\rangle x_1\rVert_2} \arccos  \langle x_1,x_2\rangle .
\end{align}
%
\subsection{The Manifold \texorpdfstring{$\mathcal P(r)$}{P(r)} of Symmetric Positive Definite Matrices}\label{A-spd}
By $\Exp$ and $\Log$ we denote the matrix exponential and logarithm defined by
\[
\Exp x \coloneqq \sum_{k=0}^\infty \frac{1}{k!} x^k,
\quad
\Log x \coloneqq \sum_{k=1}^\infty \frac{1}{k} (I-x)^k, \quad \rho(I-x) < 1.
\]
Let $\text{Sym} (r)$ denote the space of symmetric $r \times r$ matrices with (Frobenius) inner product and norm
\begin{equation} \label{inner_frob}
\langle A,B \rangle \coloneqq \sum_{i,j=1}^r a_{ij}b_{ij}, \quad \|A\| \coloneqq \left( \sum_{i,j=1} a^2_{ij} \right)^{\frac12}\!.
\end{equation}
Let $\mathcal P(r)$ be the manifold of symmetric positive definite $r \times r$ matrices.
It has the dimension $\text{dim} \, {\mathcal P}(r) = n = \frac{r(r+1)}{2}$.
The tangent space of ${\mathcal P}(r)$ at $x \in {\mathcal P}(r)$ 
is given by 
$T_x {\mathcal P}(r) = \{x\} \times \text{Sym}(r) \coloneqq \bigl\{ x^{\frac12} \eta x^{\frac12} : \eta \in \text{Sym} (r) \bigr\}$,
in particular $T_I {\mathcal P} (r) = \text{Sym} (r)$, where $I$ denotes the $r \times r$ identity matrix.
The  Riemannian metric on $T_x {\mathcal P}$ reads
\begin{equation} \label{rm_spd}
\langle \eta_1,\eta_2 \rangle_x \coloneqq
\text{tr} \bigl(\eta_1 x^{-1} \eta_2 x^{-1}\bigr) =
\bigl\langle x^{-\frac12} \eta_1 x^{-\frac12}, x^{-\frac12} \eta_2 x^{-\frac12}\bigr\rangle,
\end{equation}
where $\langle \cdot,\cdot \rangle$ denotes the matrix inner product \eqref{inner_frob}.
The geodesic distance is given by
$
d_{\mathcal P} (x_1,x_2) \coloneqq \bigl\lVert\Log\bigl(x_1^{-\frac12} x_2 x_1^{-\frac12} \bigr)\bigr\rVert,
$
and the exponential map and its inverse by
\[
\exp_x (t\eta) \coloneqq x^{\frac12} \Exp\bigl(t x^{-\frac12} \eta x^{-\frac12}\bigr) x^{\frac12},
\quad{\rm resp.,}\quad
\log_{x_1} x_2 \coloneqq x_1^\frac12 \Log \bigl(x_1^{-\frac12} x_2 x_1^{-\frac12}\bigr) x_1^{\frac12},
\]
see \cite{SH14}.
\subsection{The Manifold \texorpdfstring{\(\operatorname{SO}(3)\)}{SO(3)} of Rotation Matrices in \texorpdfstring{$\mathbb R^3$}{R3}} \label{app:so3}
The manifold of \(3\times 3\) rotation matrices is defined as
\begin{equation*}
\operatorname{SO}(3) \coloneqq \bigl\{x\in\mathbb R^{3,3}\,\big\vert\,x^\tT x=I \textrm{ and } \det{x} = 1\bigr\}.
\end{equation*}
The geodesic distance between two rotation matrices $x_1,x_2\in \operatorname{SO}(3)$ is given by
\begin{equation*}
d_{\operatorname{SO}(3)}(x_1,x_2) \coloneqq\sqrt{2}\arccos\biggr(\frac{1-\tr(x_1^\tT x_2)}{2}\biggl).
\end{equation*}
The tangential space at $x\in \operatorname{SO}(3)$ reads
\begin{align*}
T_x\operatorname{SO}(3) \coloneqq \bigl\{xv:v\in T_I\operatorname{SO}(3)\bigr\},\quad
T_I\operatorname{SO}(3) \coloneqq \bigl\{ v\in\mathbb R^{3,3}: v+v^\tT=0\bigr\}.
\end{align*}
For \(\eta \in T_x\operatorname{SO}(3)\) the exponential map at $x\in \operatorname{SO}(3)$ and  (locally) its inverse are defined as
\begin{equation*}
\exp_x(\eta) \coloneqq x\Exp\bigl(x^\tT \eta\bigr),\quad{{\rm resp.},} \quad \log_{x_1}x_2 = x_1\Log\bigl(x_1^\tT x_2\bigr).
\end{equation*}
The $\operatorname{SO}(3)$ can be parametrized in various ways.
Due to the form of our data we prefer to use quaternions for the representation 
and the similarity of $\operatorname{SO}(3)$ to the group $\mathbb S^3$, see, e.g., \cite{graef12}:
We decompose the unit quaternions 
$q = \bigl(s,v^\tT \bigr)^\tT \in \mathbb S^3$
into a real part
$s\in \mathbb R$ and a vector part $v\in\mathbb R^3$.
The multiplication of two quaternions $q_1, q_2 \in \mathbb S^3$  is given by
\begin{equation*}
q_1 \circ q_2 \coloneqq\begin{pmatrix}s_1\\v_1\end{pmatrix}\circ \begin{pmatrix}s_2\\v_2\end{pmatrix}
=\begin{pmatrix}s_1s_2-v_1^\tT v_2\\ s_2v_1+s_2v_1+v_1\times v_2\end{pmatrix},
\end{equation*}
with the conjugate quaternion \(\overline{q} \coloneqq \bigl(s,-v^\tT\bigr)^\tT\) as inverse element and unit element $(1,0,0,0)^\tT$.
The quaternions can be identified with the rotations of $\operatorname{SO}(3)$, where
the quaternions $q$ and $-q$ correspond to the same rotation.
More precisely, $\mathbb S^3$ is a double cover of $\operatorname{SO}(3)$, see, \cite[Chap. III, Sect. 10]{bredon93}:
\[
\mathbb S_*^{3} \coloneqq \mathbb S^3/\{-1,1\} \cong \operatorname{SO}(3).
\]
We work with the representative having a positive first component.
With this representation at hand, 
he Riemannian metric of \(\operatorname{SO}(3)\) can be deduced from the Euclidean metric in $\mathbb R^4$.
The geodesic distance can be written as
\begin{equation*}
d_{\mathbb S^{3}_*}(q_1,q_2)= 2 \arccos{\lvert \langle q_1, q_2\rangle\rvert}.
\end{equation*}
The exponential map of $\eta \in T_{p}{\Ss^{3}_*}$ is given by 
\begin{equation*}
	\exp_{q} \eta \coloneqq \sgn{s}\begin{pmatrix}s\\v\end{pmatrix}, \quad (s, v)^\tT \coloneqq q\cos\lVert \eta\rVert_2 + \frac{\eta}{\lVert \eta \rVert_2}\sin\lVert \eta\rVert_2
\end{equation*}
and the logarithmic map at $q_1$ of $q_2$  by 
\begin{equation*}
\log_{q_1}q_2 \coloneqq \frac{q_2-\langle q_2,q_1\rangle q_1}{\lVert q_2-\langle q_2,q_1\rangle q_1\rVert_2}\arccos\bigl(\lvert \langle q_1,q_2\rangle\rvert\bigr)\sgn\bigl(\langle q_1,q_2\rangle\bigr).
\end{equation*}
\\[2ex]
{\bf Acknowledgement:}
Funding by the DFG within the RTG GrK 1932 ``Stochastic Models for Innovations in the Engineering Sciences'', 
project area P3, and in the projects STE 571/13-1 and BE 5888/2-1 is gratefully acknowledged. Furthermore RC gratefully acknowledges support by the HKRGC Grants No. CUHK300614, CUHK2/CRF/11G, AoE/M-05/12; CUHK DAG No. 4053007, and FIS Grant No. 1907303.

\bibliographystyle{abbrv}
\bibliography{references}
\end{document}